\documentclass[12pt,a4paper]{article}
\usepackage[english]{babel}
\usepackage{amsfonts}
\usepackage{amsmath,amssymb,amsthm}
\usepackage{graphicx}

\newcommand\bdelta{{\boldsymbol \delta}}

\newcommand\ra{{\rm a}}
\newcommand\rb{{\rm b}}
\newcommand\re{{\rm e}}
\newcommand{\bb}{{\mathbf{b}}}
\newcommand{\db}{{\mathbf{d}}}
\newcommand{\ab}{\mathbf{a}}
\newcommand\bd{\mathbf{d}}
\newcommand\bc{\mathbf{c}}

\newcommand{\cb}{{\boldsymbol{c}}}

\def\bdelta{\mbox{\boldmath $\delta$}}

\newcommand\cA{{\mathcal A}}

\newcommand\cT{{\mathcal T}}

\newcommand\cM{{\mathcal M}}

\newcommand{\fpak}{\cA}

\newcommand{\ZZ}{\mathbb{Z}}
\newcommand{\RR}{\mathbb{R}}
\newcommand{\NN}{\mathbb{N}}
\newcommand{\CC}{\mathbb{C}}

\newcommand{\n}{\noindent}
\newcommand{\eps}{\varepsilon}

\newcommand\ie{{\it\thinspace i.e.}\ }

\newtheorem{theorem}{Theorem}
\newtheorem{proposition}{Proposition}
\newtheorem{lemma}{Lemma}
\newtheorem{corollary}{Corollary}
\newtheorem{remark}{Remark}
\newtheorem{example}{Example}
\newtheorem{definition}{Definition}
\newtheorem{conjecture}{Conjecture}

\title{Regularity of non-stationary subdivision: a matrix approach\thanks{The authors are grateful to
the Mathematical Institute at Oberwolfach for supporting the Research In Pairs Program in 2013 and to the Italian G.N.C.S.}}

\author{M. Charina\thanks{University of Vienna, Austria, maria.charina@univie.ac.at.}, \  C. Conti\thanks{DIEF-University of Florence, Italy,
 costanza.conti@unifi.it.}, \  N. Guglielmi\thanks{University of L'Aquila, Italy, guglielm@dm.univaq.it.} \ and 
V. Protasov \thanks{Moscow State University, Russia, v-protassov@yandex.ru.}}

\begin{document}
\maketitle

\begin{abstract}
In this paper, we study scalar multivariate non-stationary
subdivision schemes with integer dilation matrix~$M$
and present a unifying, general approach for checking their
convergence and for determining their H\"older regularity
(latter in the case~$M = mI,  m \ge 2$). The
combination of the concepts of asymptotic similarity and
approximate sum rules allows us to link stationary and
non-stationary settings and to employ recent advances in methods
for exact computation of the joint spectral radius. As an
application, we prove a recent conjecture by N. Dyn et al. on the H\"older
regularity of the generalized Daubechies wavelets. We illustrate
our results with several simple examples.
\end{abstract}

\noindent {\bf Keywords:} multivariate non-stationary subdivision schemes,
convergence, H\"older regularity, approximate sum rules, asymptotic similarity,
joint spectral radius, generalized Daubechies wavelets

\noindent{\bf Classification (MSCS): 65D17, 15A60, 39A99}

\pagestyle{myheadings}
\thispagestyle{plain}

\section{Introduction}\label{sec:intro}
We provide a general, unifying method for convergence and
regularity analysis of multivariate non-stationary, i.e.
level-dependent, subdivision schemes with an integer dilation
matrix $M$ whose eigenvalues are all larger than $1$ in the absolute value. It has been
believed until recently that the joint spectral radius approach,  successfully used for the
regularity analysis of stationary subdivisions, is not applicable in the non-stationary setting.
Our results dismiss this belief. We show that the joint spectral radius techniques are
applicable for all non-stationary schemes that satisfy two mild assumptions: all level-dependent
masks have the same bounded support and satisfy the so-called {\em approximate sum rules}. We show that
the approximate sum rules are ``almost necessary'' for convergence and regularity of
non-stationary schemes. We derive sharp
criteria for convergence of non-stationary schemes in the spaces $C^{\ell}\, , \ \ell \ge 0$,
and, in the case $M = mI, m \ge 2$, obtain a formula for the H\"older exponent of their limit functions.
The application of our results allows us, e.g. to determine the H\"older regularity  of
the generalized  Daubechies wavelets and, thus, prove the recent conjecture by N. Dyn et al.~\cite{DKLR}. In this paper we focus mainly on
subdivision schemes, although our results are applicable to all non-stationary refinable functions,
in particular, the ones used for constructions of non-stationary wavelets.

Subdivision schemes are linear iterative algorithms that interpolate or
approximate data on a given polygonal mesh.  From starting data, such
schemes  repeatedly compute local, linear weighted averages of
sequences of real numbers or point in $\RR^s$, $s=2,3$. The weights are real
numbers that define the so-called subdivision mask of the scheme. The scheme
converges, if the subdivision recursion  generates a continuous limit
function from every starting data sequence.  The first subdivision
scheme is the corner cutting algorithm by de Rham \cite{Rham} that generates smooth curves 
from the given vertices of a polygon in $\RR^2$.

Subdivision schemes are simple to implement,
intuitive in use, and possess many other nice   properties (linearity, shift-invariance, etc).
This motivates their wide popularity in modeling freeform
curves and surfaces and in computer animation. The potential of subdivision schemes
has recently also become apparent in the context of isogeometric analysis, a modern computational
approach that integrates finite element analysis into conventional CAD systems. Thus, in the last ten years,
there has been an increase of interest in subdivision schemes. The main questions when analyzing
any scheme are its convergence and the regularity of its limit functions. Both of these questions
can be answered effectively using the {\em matrix approach} that reduces these questions to computation
of the {\em joint spectral radius} of a special, compact set of square matrices.

Non-stationary subdivision schemes were introduced to enrich the class of
limit functions of stationary schemes and have very
different and distinguished properties. Indeed, it is well-known that stationary subdivision schemes are not capable
of generating circles, ellipses, or, in general, of generating exponential (quasi-) polynomials $x^\gamma e^{\lambda \cdot x}$,
$x \in \RR^s$, $\gamma \in  \NN_0^s$, $\lambda \in \CC^s$, while non-stationary schemes generate function spaces
that are much richer and include exponential polynomials as well as exponential $B-$splines (see e.g. \cite{BR, Unser1, Dyn2, JiaLei, Han, LycheMaz, Ron_exp_box_splines}). This generation property is
important in several applications, e.g. in biological imaging
\cite{Unser2, Unser3}, geometric design \cite{Yoon1, Romani, W} and in
isogeometric analysis (see \cite{Burkhart10,COS00,CSAOS02,ISO} and references therein).  The interest in non-stationary subdivision schemes is also due to the fact that they include Hermite schemes that do
not only model curves and surfaces, but also their gradient
fields. Such schemes are used in geometric modeling and
biological imaging, see e.g. \cite{Manni1, LycheMer1, LycheMer2,
Manni3, Unser6}. Additionally, multi-resolution analysis based
on any stationary subdivision scheme uses the same filters at each
level of the decomposition and reconstruction. On the contrary,  non-stationary
wavelet and frame constructions are level adapted and more
flexible, see e.g. \cite{DKLR, GoodmanLee, Han, Unser4, Unser5}.
Unfortunately, in practice, the use of subdivision is mostly  restricted to the  class of stationary subdivision
schemes. One reason for that is the lack of general methods for their analysis, especially  methods for their
convergence and regularity analysis. This motivates our study.

The main difficulty is that the matrix approach cannot be directly extended to
the non-stationary setting. In the stationary case, the H\"older regularity of
subdivision limits is derived from the joint spectral radius of a finite set of
linear operators which are restrictions of {\em transition matrices} of the subdivision
scheme to their special linear subspace. In the non-stationary case, one deals with a sequence of transition matrices and a
sequence of their corresponding linear subspaces. Both  may not converge a priori.
Deep analysis of such sequences allows us to prove that the sequence of such linear
subspaces does possess a limiting subspace (provided the scheme converges). Moreover,
as in the stationary case, we show how to express the H\"older regularity of non-stationary
subdivision in terms of the joint spectral radius of the limit points of the sequence of
transition matrices restricted to this limiting linear subspace.
Both results provide a powerful tool for analysis of  non-stationary subdivision schemes.
Several numerical examples demonstrate the efficiency of our method.

Finally, note  that there is another class of non-stationary schemes that can generate $C^\infty$ limits
(Rvachev-type functions) with bounded support, see \cite{DL2002}.
The trade-off is that the supports of their level-dependent subdivision
masks grow from level to level of the subdivision recursion.
Our approach for regularity analysis is based on computations of the joint spectral
radius of a compact set of matrices. Therefore, naturally, it does not apply
to Rvachev-type schemes, since, the corresponding matrix sets are unbounded. For analysis and applications of
Rvachev-type schemes we refer the reader, for example, to \cite{CohenDyn, DL2002, Han2}.

\subsection{Framework}

\smallskip \n Let $M=mI$, $m \ge 2$. Given an initial set of  data $\cb^{(1)}:=\{ c^{(1)}(\alpha) \in \RR,\
\alpha\in \ZZ^s\}$ a subdivision scheme iteratively constructs a sequence of progressively denser data by means of \emph{local refinement rules} which are based on the sequence of \emph{subdivision operators}
$\{S_{\mathbf{a}^{(k)}},\ k\ge 1\}$. The subdivision operators
$S_{\mathbf{a}^{(k)}}: \ell(\ZZ^s) \rightarrow \ell(\ZZ^s)$ are
linear operators and map coarser sequences $\cb^{(k)} \in
\ell(\ZZ^s)$ into finer sequences $\cb^{(k+1)} \in \ell(\ZZ^s)$
via the rules
\begin{equation}\label{def:refequation}
  \cb^{(k+1)}:=S_{\mathbf{a}^{(k)}} \cb^{(k)}, \ \  S_{\mathbf{a}^{(k)}}
  \cb^{(k)}(\alpha):=\sum_{\beta \in\ZZ^s} \ra^{(k)}(\alpha-M\beta) c^{(k)}(\beta), \ \ k \ge 1, \ \ \alpha \in \ZZ^s\,.
\end{equation}
The \emph{masks} $\{\mathbf{a}^{(k)},\ k\ge 1\}$ are sequences of real numbers
$\mathbf{a}^{(k)}:=\{\ra^{(k)}(\alpha) \in \RR, \ \alpha \in \ZZ^s\}$
and we
assume that they  have bounded supports in $\{0,\ldots,N\}^s$ with $N \in
\NN$. Such schemes are called either \emph{level-dependent}, or
\emph{non-stationary} or \emph{non-homogeneous}. Here we use the
term \emph{non-stationary} and denote these type of subdivision
schemes by the corresponding collection of subdivision operators
$\{S_{\ab^{(k)}}, \ k\ge 1 \}$. A subdivision scheme whose
refinement rules are level independent is said to be
\emph{stationary} (see, \cite{CaravettaDahmenMicchelli}, for
example) and, for all $k \ge 1$, is defined by the same sequence
$\mathbf{a}:=\{\ra(\alpha) \in \RR,\ \alpha\in \ZZ^s\}$ of refinement
coefficients, i.e. $\mathbf{a}^{(k)}=\mathbf{a}$, $k \ge 1$. The
corresponding subdivision scheme is therefore  denoted by $S_\ab$.
There is a multitude of results on convergence and regularity of stationary subdivision
schemes in the literature (for example see \cite{CaravettaDahmenMicchelli, CharinaContiSauer2005,
ChenJiaRie, J95, H03, MicSauer97} and references therein). These results rely on polynomial generation and reproduction properties
of subdivision operators and employ the so-called restricted spectral radius or the joint spectral
radius techniques. It has been believed until recently that these two concepts have no immediate application
in the non-stationary setting. The reason for this belief is that convergent non-stationary schemes do not necessarily
generate or reproduce any polynomial spaces, see e.g. \cite{ContiGRomani}.

In this paper, we make use of  the concepts of \emph{approximate sum rules} and \emph{asymptotic
similarity} to link stationary and non-stationary settings and show how to employ
the \emph{joint spectral radius} for smoothness analysis of non-stationary schemes.
This allows us to provide a general and unifying approach for the analysis of convergence and
regularity of a vast majority of non-stationary subdivision schemes. Our results
generalize the existing well-known methods in \cite{ContiDynManniMazure13, DynLevin, DynLevinYoon},
which only  allow us to check convergence and H\"older regularity of special instances of
non-stationary schemes. In fact, the sufficient conditions in \cite{DynLevin} are based on the concept of \emph{asymptotic equivalence}
which we recall in the following  Definition \ref{def:asymptotic_equivalence},
where $E$ is a set of representatives of $\ZZ^s /
M \ZZ^s$, i.e. $E \simeq \{0,\ldots,m-1\}^s$.

\begin{definition} \label{def:asymptotic_equivalence}  Let $\ell \ge 0$.
Two non-stationary schemes $\{S_{\mathbf{a}^{(k)}}, \ k \ge 1\}$ and $\{S_{\mathbf{b}^{(k)}}, \ k \ge 1\}$ are called \emph{asymptotically equivalent (of order $\ell$)}, if they satisfy
\begin{equation}\label{def:AsEq}
\sum_{k=1}^\infty m^{k \ell}
\|S_{\mathbf{a}^{(k)}}-S_{\mathbf{b}^{(k)}}\|_\infty <\infty,
\ \hbox{for} \
\|S_{\mathbf{a}^{(k)}}\|_\infty:=\max_{\varepsilon\in
E}\left\{\sum_{\alpha\in
\ZZ^s}|\ra^{(k)}(M\alpha+\varepsilon)|\right\}\,.
\end{equation}
\end{definition}
In the case of $M=2I$ and under certain additional assumptions on
the schemes $\{S_{\mathbf{a}^{(k)}}, \ k \ge 1\}$ and
$\{S_{\mathbf{b}^{(k)}}, \ k \ge 1\}$, the method in
\cite{DynLevin} allows us to determine the regularity of
$\{S_{\mathbf{a}^{(k)}}, \ k \ge 1\}$ from the known regularity of
the asymptotically equivalent scheme $\{S_{\mathbf{b}^{(k)}}, \ k \ge
1\}$. In \cite{DynLevinYoon}, in the univariate binary case,
the authors relax the condition of asymptotic equivalence.
They require that the $D^j$-th derivatives of the symbols
$$
 a_*^{(k)}(z):=\sum_{\alpha \in \ZZ} \ra^{(k)}(\alpha)z^\alpha, \quad z \in \CC \setminus \{0\}, \quad k\ge 1,
$$
of the non-stationary scheme $\{S_{\ab^{(k)}}, \ k\ge 1 \}$ satisfy
\begin{equation} \label{cond:ConditionA}
  |D^j a_*^{(k)}(-1)| \le C \, 2^{-(\ell+1-j)k}, \quad j=0,\ldots,\ell,
  \quad \ell \ge 0,\quad C \ge 0,
\end{equation}
and, additionally, assume that the non-stationary scheme is
asymptotically equivalent (of order $0$) to some stationary
scheme.  The conditions in \eqref{cond:ConditionA} can be seen as
a generalization of the so-called sum rules in \eqref{def:sumrules}. In the stationary case, sum rules are
necessary for smoothness of subdivision, see e.g \cite{Cabrelli,
CaravettaDahmenMicchelli, JetterPlonka, JiaJiang}.

\begin{definition}  Let $\ell \ge 0$.
The symbol $a_*(z)$, $z \in \CC^s \setminus \{0\}$, satisfies sum
rules of order $\ell+1$ if
\begin{equation} \label{def:sumrules}
 {a}_*(1)=m^s \quad\hbox{and}\quad
 \max_{|\eta| \le \ell}\  \max_{\epsilon \in \Xi \setminus
\{1\}} | D^\eta a_*(\epsilon)|=0\,.
\end{equation}
\end{definition}

\noindent In the above definition, $\Xi:=\{e^{-i\frac{2\pi}{m}\varepsilon}, \ \varepsilon \in E\}$ and
$D^\eta$, $\eta \in \NN_0^s$, denotes  the $\eta-$th directional derivative.

\smallskip \noindent In the spirit of (\ref{cond:ConditionA}),  in this paper, we present
a generalization of the notion of sum rules which we call
\emph{approximate sum rules}.

\begin{definition} \label{def:approx_sum_rules} Let $\ell \ge 0$. The sequence of
symbols $\{a^{(k)}_*(z),\ \ k \ge 1\}$ satisfies
approximate sum rules of order $\ell+1$, if
  \begin{equation}\label{def:delta_k}
 \mu_k:=|{a}_*^{(k)}(1)-m^s| \quad \hbox{and}\quad
 \delta_k\ :=\ \max_{|\eta| \le \ell} \max_{\epsilon \in \Xi \setminus \{1\}} |{m^{-k|\eta|}  D^\eta {a}_*^{(k)}(\epsilon)|}
  \end{equation}
satisfy
  \begin{equation}\label{proporties:delta_k}
 \sum_{k=1}^{\infty} \mu_k < \infty \quad \hbox{and}\quad  \sum_{k=1}^{\infty} m^{\, k\ell}\, \delta_k < \infty\,.
  \end{equation}
\end{definition}

\smallskip \noindent We call the sequence $\{\delta_k, \ k \ge 1\}$ \emph{sum rule defects}.
If the sequences $\{\mu_k, \ k \ge 1\}$ and $\{\delta_k, \ k \ge
1\}$ are zero sequences, then the symbols of the corresponding
non-stationary scheme satisfy sum rules of order $\ell+1$.

\smallskip \noindent  Note that, even in the univariate binary case, the assumption on
$\{\delta_k, \ k \ge 1\}$ in \eqref{proporties:delta_k}, i.e.
\begin{equation} \label{def:deltak_uni}
 \sum_{k=1}^\infty 2^{\ell k} \delta_k < \infty, \quad
 \delta_k:=\ \max_{j \le \ell} {2^{-k\, j} | D^j {a}_*^{(k)}(-1)|},
\end{equation}
is less restrictive than the decay condition on $\{\delta_k, \ k
\ge 1\}$ in \eqref{cond:ConditionA}.  In Theorem \ref{th:neccesary_conditions}, we showed that approximate sum rules
are close to being necessary conditions for regularity of
non-stationary schemes, i.e. even in the univariate binary setting, the sum rules
defects~$\{\delta_k, \ k\ge 1\}$ must decay faster than $2^{\, -\ell k}$, if the
limit functions of the scheme are $C^\ell$. Indeed, in the
binary univariate case, we show that under assumption of
asymptotical similarity (see Definition
\ref{def:asympt-similarity}) to a stationary scheme  whose  basic refinable function is
stable, the $C^\ell-$regularity of the non-stationary scheme
implies that the sum rules defects~$\{\delta_k, \ k\ge 1\}$ must
decay faster than $2^{\, -\ell k}$. Clearly,  there is  still a gap between the corresponding
necessary condition $\displaystyle \lim_{k \to \infty} 2^{-\ell k}\delta_k =
0$ and one of the sufficient conditions $\displaystyle \sum_{k \in
\NN} 2^{-\ell k} \delta_k < \infty$. See also Example \ref{ex:no_approx_sum_rules}. Moreover,  in
\cite{CRY}, the authors proved that this decay rate of the sum rules defects is necessary for generation
of  $\ell$ linearly independent functions from $\{x^\gamma e^{\lambda x}, \ \gamma \in \NN_0, \ \lambda \in \CC\}$.
This resembles the stationary setting and motivates our multivariate convergence and
smoothness analysis of non-stationary schemes.

\smallskip \noindent In \cite{ContiDynManniMazure13}, in the univariate binary
non-stationary setting, milder sufficient conditions than asymptotic equivalence  are
essentially derived under the assumptions that the scheme
$\{S_{\mathbf{a}^{(k)}},\ \ k \ge 1\}$ is \emph{asymptotically
similar} to a suitable non-stationary scheme
$\{S_{\mathbf{b}^{(k)}},\ \ k \ge 1\}$, \ie
$\lim\limits_{k\rightarrow
\infty}\|\mathbf{a}^{(k)}-\mathbf{b}^{(k)}\|_\infty=0\,$, and both
satisfy sum rules of order $1$. Here we generalize the notion of
asymptotic similarity making use of the following concept of
\emph{set of limit points} of a sequence of masks.

\begin{definition} \label{def:set_of_limit_points} For the mask sequence
$\{\mathbf{a}^{(k)},\ \ k \ge 1\}$ we denote by $\cA$ the set of
its limit points, i.e. the set of masks $\mathbf{a}$ such that
$$ \mathbf{a}\in \cA,\quad \hbox{if} \quad \exists \{k_n,\ n\in \NN \}\ \ \mbox{such that}\
 \ \lim_{n\rightarrow\infty}\mathbf{a}^{(k_n)}=\ab\,.
$$
\end{definition}
The following definition of \emph{asymptotic similarity}
generalizes the one given in \cite{ContiDynManniMazure13}. This
notion allows us to relate the properties of non-stationary
subdivision schemes to the corresponding properties of the
stationary masks in $\cal A$.

\begin{definition} \label{def:asympt-similarity}
Two non-stationary schemes $\{S_{\mathbf{a}^{(k)}},\ \ k \ge 1\}$ and   $\{S_{\mathbf{b}^{(k)}},\ \ k \ge 1\}$ are called \emph{asymptotically similar}, if their sets of limit points coincide.
\end{definition}

\subsection{Summary of the results}\label{subsec:summary}
For the reader's convenience,  we summarize here the main results
presented in this paper. The details are given in \S
\ref{sec:convergence_regularity}.

\noindent In the rest of the paper we assume that the symbols
$\{{a}_*^{(k)}(z), \ k \ge 1\}$ satisfy approximate sum rules and
are re-scaled in such a way that ${a}_*^{(k)}(1)=m^s$, $k \ge 1$.
In this case $\mu_k$ in \eqref{def:delta_k} are equal to zero for
all $k \ge 1$ and do not affect our convergence and regularity
analysis. On the contrary, if the sequence $\{\mu_k, \ k \ge 1\}$
is not summable, then such a re-scaling can change the properties
of the scheme, see Example \ref{ex:no_approx_sum_rules}.

\smallskip \noindent
One of our results states that even in the univariate case
approximate sum rules are close to being necessary for
convergence and smoothness of non-stationary subdivision schemes.

\begin{theorem} \label{th:neccesary_conditions} Let $\ell \ge 0$.
Assume that a univariate binary
subdivision scheme $S_\mathbf{a}$ is convergent and its basic
refinable limit function is stable. Assume, furthermore, that $
\displaystyle
 \mathbf{a}=\lim_{k \rightarrow \infty} \mathbf{a}^{(k)}
$ and the non-stationary subdivision scheme
$\{S_\mathbf{a}^{(k)},\ k\ge 1\}$ converges to $C^\ell$ limit
functions. Then $\displaystyle \lim_{k \to \infty} 2^{\, \ell k}
\delta_k = 0$ for $\{\delta_k, \ k \ge 1\}$ in \eqref{def:deltak_uni}.
\end{theorem}

\medskip \noindent The proof of Theorem \ref{th:neccesary_conditions} is given
in \S \ref{sec:necessary}. Thus, if the scheme converges to a $C^{\ell}$ limit function, then
the sum rule defects $\{\delta_k, \ k \ge 1\}$ decay faster than  $2^{\, - \ell k}$, i.e.,
 satisfy $\delta_k \, = \, o(2^{\, - \ell k})$.
This does not imply the approximate sum rules, i.e. $\sum_{k \in \NN}\delta_k 2^{\, \ell k} < 0$,
but is close to this condition. Thus, Theorem~\ref{th:neccesary_conditions} indicates that
approximate sum rules is a natural assumption for convergence to a $C^{\ell}$ limit,
and cannot be relaxed by much.

\smallskip \noindent In the stationary case,  the H\"older regularity of the
subdivision limits, as well as the rate of convergence of the
corresponding subdivision scheme $S_{\mathbf{a}}$, are determined
explicitly in terms of the joint spectral radius of the set of
certain square matrices which are derived from the subdivision
mask $\mathbf{a}$ and depend on the order of sum rules satisfied
by $a_*(z)$. Since, in the non-stationary setting, one cannot
assume that all subdivision symbols $\{a^{(k)}_*(z),\ k\ge 1\}$
satisfy sum rules, see \cite{CharinaContiRomani13, ContiRomani11},
the concept of the joint spectral radius is not directly
applicable and has no straightforward generalization. For this
reason, in Theorem \ref{th:main_Hoelder}, we establish a link
between stationary and non-stationary settings via the set $\cA$
of limit points of $\{\mathbf{a}^{(k)},\ k\ge 1\}$ and provide
sufficient conditions for $C^\ell-$convergence, $\ell \ge 0$, and
H\"older regularity of non-stationary schemes. Under
$C^\ell-$convergence we understand the convergence of subdivision
in the norm of $C^\ell(\RR^s)$, see Definition
\ref{def:Cellconvergence} in \S \ref{sec:background}. Note
that $C^0-$convergence is the usual convergence of subdivision in
$\ell_\infty$ norm and $C^\ell-$convergence implies the
convergence of the scheme to $C^\ell$ limit functions, but not
vice versa, see Definition \ref{def:convergenceTOCell} in \S
\ref{sec:background}. As in the stationary setting, each mask in
the limit set $\cA$ determines a set of transition matrices, see e.g. \eqref{def:Tepsilon_a}. We
denote the collection of the restrictions of all these transition
matrices to a given finite dimensional difference subspace
$V_\ell$ by $\cT_{\cA}|_{V_\ell}$, see \S \ref{sec:background} for more details. Theorem \ref{th:main_Hoelder}
states that $C^\ell-$convergence and H\"older regularity of
non-stationary schemes is determined by the joint spectral radius
$\rho_\cA$ of this collection $\cT_{\cA}|_{V_\ell}$.

\begin{theorem} \label{th:main_Hoelder} Let $\ell \ge 0$ and
$\{\delta_k, \ k\ge 1\}$ be defined in \eqref{def:delta_k}. Assume
that the symbols of $\{ S_{{\mathbf a}^{(k)}}, \ k \ge 1\}$
satisfy approximate sum rules of order $\ell+1$ and
$\rho_\cA:=\rho\left(\cT_{\cA}|_{V_\ell}\right)<m^{-\ell}$, where
$\cA$ is the set of limit points of $\{\mathbf{a}^{(k)},\ k\ge
1\}$. Then the non-stationary scheme $\{ S_{{\mathbf a}^{(k)}}, \
k \ge 1\}$ is $C^{\ell}-$convergent and the H\"older exponent
$\alpha$ of its limit functions satisfies
$$
 \alpha \ \ge \ \min \, \Bigl\{\, -\log_m \rho_\cA\, , \,
 - \limsup\limits_{k \to \infty}\frac{\log_m \delta_k}{k}\, \Bigr\}.
$$
\end{theorem}

\smallskip \noindent The proof of Theorem \ref{th:main_Hoelder} is given in \S
\ref{subsec:regularity}. Thus, in the non-stationary case, the smoothness of the limit function
depends not only on the joint spectral radius of the  matrices in $\cal A$, but also
on the rate of decay of the sum rules defects~$\{\delta_k, \ k \ge 1\}$.

\smallskip \noindent Note that, as in the stationary case, the order of
approximate sum rules satisfied by the symbols of a non-stationary
scheme can be much higher than its regularity.

\smallskip \noindent There are several immediate important consequences of Theorem
\ref{th:main_Hoelder}
 that  generalize the corresponding results in
\cite{ContiDynManniMazure13, DynLevin, DynLevinYoon}. For example
the following Corollary extends the results
in \cite{ContiDynManniMazure13} with respect to the dimension of
the space, the regularity of the limit functions and the more
general notion of asymptotic similarity given in Definition
\ref{def:asympt-similarity}.

\begin{corollary}  Let $\ell \ge 0$. Assume that the symbols of the scheme
$\{ S_{{\mathbf a}^{(k)}}, \ k \ge 1\}$   satisfy sum rules of
order $\ell+1$ and
$\rho_\cA:=\rho\left(\cT_{\cA}|_{V_\ell}\right)<m^{-\ell}$, where
$\cA$ is the set of limit points of $\{ {\mathbf a}^{(k)}, \ k \ge
1\}$. Then any other asymptotically similar scheme $\{S_{{\mathbf
b}^{(k)}}, \ k \ge 1\}$ whose symbols  satisfy sum rules of order
$\ell+1$ is $C^\ell-$convergent and the H\"older exponent of its
limit functions is $ \alpha \ \ge -\log_m \rho_\cA$.
\end{corollary}

\smallskip \noindent Theorem \ref{th:main_Hoelder} provides a  lower bound for the
H\"older exponent of the subdivision limits, whereas the next
result allows us to determine its exact value, under slightly more restrictive
assumptions.

\begin{theorem} \label{th:stable_case}  Let $\ell \ge 0$. Assume that a
stationary scheme $S_{\mathbf{a}}$ is $C^\ell-$convergent with the
stable refinable basic limit function $\phi$ whose H\"older
exponent $\alpha_\phi$ satisfies $\ell \le \alpha_\phi<\ell+1$. If
the symbols of the scheme $\{ S_{{\mathbf a}^{(k)}}, \ k \ge 1\}$
satisfy approximate sum rules of order $\ell+1$, $\displaystyle
\lim_{k \rightarrow \infty} \mathbf{a}^{(k)}=\mathbf{a}$ and,
additionally
$$
 \limsup_{k \rightarrow \infty} \delta_k^{1/k} < \rho_\mathbf{a}:=
 \rho(\{T_{\varepsilon, \mathbf{a}}|_{V_\ell}, \ \varepsilon \in
 E\}),
$$
then the scheme $\{ S_{{\mathbf a}^{(k)}}, \ k \ge 1\}$ is
$C^\ell-$convergent and the H\"older exponent of its limit
functions is also $\alpha_\phi$.
\end{theorem}

\smallskip \noindent The proof of Theorem \ref{th:stable_case} is given in
\S \ref{subsec:rapidly}. An important special class of non-stationary
schemes that satisfy assumptions of Theorem \ref{th:stable_case} are the schemes
whose symbols satisfy sum rules of order $\ell+1$, see Corollary
\ref{th:cor_stable_case} in \S \ref{subsec:rapidly}.

\smallskip \noindent The main application of Theorem \ref{th:stable_case} is the proof of the following
conjecture by N. Dyn et al. stated in \cite{DKLR} in 2014.

\begin{conjecture} [\cite{DKLR}] \label{conj}  The H\"older
regularity of every generalized Daubechies type wavelet is equal to the H\"older regularity of the
corresponding classical Daubechies wavelet.
\end{conjecture}

\noindent We prove this conjecture in Theorem \ref{th50} and
compute some of the corresponding H\"older exponents, see \S \ref{subsubsec:Daub_wavelets}.

\medskip \noindent This  paper is organized as follows. In \S
\ref{sec:background}, we summarize important known fact about
stationary and non-stationary subdivision schemes. The proofs of
the results stated in \S \ref{subsec:summary} are given in
\S \ref{sec:convergence_regularity}. In particular, in
\S \ref{subsec:convergence}, we provide sufficient
conditions for convergence of non-stationary subdivision schemes
whose symbols satisfy assumptions of Theorem \ref{th:main_Hoelder}
with $\ell=0$. The argument in the proof of the corresponding
Theorem \ref{teo:convergence} is actually independent of the
choice of the dilation matrix $M$. For that reason we give a
separate proof of convergence and, then, in \S
\ref{subsec:regularity}, present the proof of the more general
statement of Theorem \ref{th:main_Hoelder}. In \S
\ref{subsec:rapidly}, we give the proof of Theorem
\ref{th:stable_case}. We illustrate our convergence and regularity
results with several deliberately simple examples in \S
\ref{subsec:convergence_examples}. There we also prove
Conjecture \ref{conj} formulated in \cite{DKLR} about the regularity of
generalized Daubechies wavelets.  Next, in \S \ref{sec:necessary}, we
prove the necessary conditions stated in Theorem \ref{th:neccesary_conditions}.

\section{Background and preliminary definitions}\label{sec:background}In this section we recall well-known properties of subdivision
schemes. We start by defining convergence and H\"older regularity
of non-stationary and, thus, also of stationary subdivision
schemes. We would like to distinguish between the following two different types of
convergence, both being investigated in the literature on stationary and non-stationary
subdivision
schemes. We denote by $\ell_\infty(\ZZ^s)$ the space of all scalar sequences
$\cb=\{c(\alpha), \ \alpha \in \ZZ^s\}$ indexed by $\ZZ^s$ and such that
$
 \|\cb\|_{\ell_\infty}:=\sup_{\alpha \in \ZZ^s} |c(\alpha)| < \infty.
$

\begin{definition} \label{def:convergenceTOCell}
A subdivision scheme $\{S_{\ab^{(k)}}, \  k\ge 1\}$ \emph{converges to
$C^\ell$ limit functions}, if
for any initial sequence $\cb \in \ell_\infty(\ZZ^s)$, there exists the
\emph{limit function} $g_{\cb} \in C^\ell(\RR^s)$ (which is
nonzero for at least one nonzero sequence $\cb$) such that
\begin{equation}\label{eq:subdivisionlimit}
  \lim_{k \to\infty}  \left\| g_{\cb}(M^{-k}\alpha)- S_{\ab^{(k)}} S_{\ab^{(k-1)}} \cdots S_{\ab^{(1)}} \cb(\alpha) \right\|_{\ell_\infty}= 0.
\end{equation}
\end{definition}

\noindent In the next Definition \ref{def:Cellconvergence} we consider a stronger type of
convergence, the so-called $C^\ell-$convergence of subdivision.
Note that both types of convergence coincide in the case $\ell=0$.
In Definition \ref{def:Cellconvergence} we make use of the concept
of a test function (see, for example \cite{DM97}). To
define this concept we need to recall the following properties of
the test functions.

\begin{definition} Let $\ell \ge 0$.
A compactly supported summable function $f$ satisfies
Strang-Fix conditions of order  $\ell+1$, if its Fourier transform
$\hat{f}$ satisfies
$$
 \hat{f}(0)=1, \quad D^\mu \hat{f}(\alpha)=0, \quad \alpha \in \ZZ^s \setminus \{0\}, \quad \mu \in \NN_0^s, \quad
 |\mu|<
\ell+1.
$$
\end{definition}

\begin{definition}
A compactly supported $f \in L_\infty(\RR^s)$ is
stable, if there exists $0<C_1 \le C_2 <\infty$ such that for all $\cb \in \ell_\infty(\ZZ^s)$,
$$
 C_1 \|\cb\|_{\ell_\infty} \le \big\| \sum_{\alpha \in \ZZ^s} c(\alpha)f(\cdot-\alpha)
 \big\|_\infty \le C_2\|\cb\|_{\ell_\infty}\,.
$$
\end{definition}

\smallskip \noindent By \cite[p.24]{CaravettaDahmenMicchelli}, this type of stability
is equivalent to $\ell_\infty$ linear independence of integer
shifts of $f$.  The function $f$ is called \emph{a test function}, if it is
sufficiently smooth, compactly supported, stable and satisfies
Strang-Fix conditions of order $\ell+1$. Possible examples of the
test functions $f$ are tensor-product box splines.

\begin{definition} \label{def:Cellconvergence}
A subdivision scheme $\{S_{\ab^{(k)}}, \  k\ge 1\}$ is
\emph{$C^\ell$-convergent}, if for any initial sequence $\cb \in \ell_\infty(\ZZ^s)$
there exists the
 limit function $g_\cb \in C^\ell(\RR^s)$ such that for any test function
$f \in C^\ell(\RR^s)$
\begin{equation}\label{eq:C^l_convergence}
  \lim_{k \to\infty} \Big \|  g_{\cb}(\cdot) - \sum_{\alpha \in \ZZ^s} S_{\ab^{(k)}}
 S_{\ab^{(k-1)}}  \cdots  S_{\ab^{(1)}}\bc(\alpha) f(M^k\cdot  - \alpha) \Big\|_{C^\ell}=0.
\end{equation}
\end{definition}

\noindent Note that, for $C^\ell-$convergence, it suffices to check
\eqref{eq:C^l_convergence} just for one test function $f$. In this paper, we also investigate the H\"older
regularity of subdivision limits.

\begin{definition} The \emph{H\"older regularity} of the $C^0-$convergent
scheme $\{S_{\ab^{(k)}}, \  k\ge 1\}$ is $\alpha=\ell+\zeta$, if
$\ell$ is the largest integer such that $g_{\cb} \in
C^\ell(\RR^s)$ and  $\zeta$ is the supremum of $\nu \in [0,1]$ such that
$$
 \max_{\mu \in \NN_0^s,  |\mu|=\ell} |D^\mu g_{\cb}(x)-D^\mu g_{\cb}(y)| \le
  |x-y|^\nu, \quad x,y \in \RR^s.
$$
We call $\alpha$ the \emph{H\"older
exponent} of the limit functions of $\{S_{\ab^{(k)}}, \ k\ge 1\}$.
\end{definition}

\smallskip \noindent Instead of studying the regularity of all limit
functions of a $C^0-$convergent subdivision scheme, one usually
restricts the analysis to the so-called \emph{basic limit
functions}, which are defined as follows. Let
$\bdelta:=\{\delta(\alpha)=\delta_{0,\alpha},\ \alpha\in \ZZ^s\}$,
where $\delta_{0,\alpha}$, $\alpha \in \ZZ^s$, is the Kronecker
delta symbol, i.e., $\delta_{0,0}=1$ and zero otherwise. The
compactly supported \emph{basic limit functions} $\phi_k$
generated from the initial sequence $\bdelta$ are given by
$$
 \phi_{k}:=\lim_{\ell \to\infty}  S_{\ab^{(k+\ell)}}  \cdots S_{\ab^{(k+1)}}
 S_{\ab^{(k)}} \bdelta, \quad k \ge 1.
$$
An interesting fact about convergent non-stationary schemes is
that the compactly supported \emph{basic limit functions} $\phi_k$
are mutually refinable, i.e., they satisfy the functional
equations
\begin{equation}\label{eq:ns-refinability}
    \phi_{k}=\sum_{\alpha \in \ZZ^s} \ra^{(k)}(\alpha) \phi_{k+1}(M\cdot-\alpha),\quad k
\ge 1,
\end{equation}
where $\{\ra^{(k)}(\alpha),\ \alpha\in \ZZ^s\}$ is the $k$-level
subdivision mask. We remark that, without loss of generality, to
study convergence and regularity of a non-stationary subdivision
scheme it suffices to study the continuity and the H\"older
regularity of the function $\phi_1$. This fact is shown in the next lemma (see also \cite{P3}).

\begin{lemma} \label{lem:alpha_k=alpha_1} Let $\alpha_{\phi_k}$ be the
H\"older exponent of $\phi_k$, $k \ge 1$. If $\ra^{(k)}(0) \not =0$,
then $\alpha_{\phi_k}=\alpha_{\phi_1}$ for all $k \ge 1$.
\end{lemma}
\begin{proof} Let $k \ge 1$. Due to \eqref{eq:ns-refinability} and the compact support of
the mask $\mathbf{a}^{(k)}$,
we have $\alpha_{\phi_{k}} \ge \alpha_{\phi_{k+1}}$ and it
suffices to show that $\alpha_{\phi_{k+1}} \ge \alpha_{\phi_{k}}$.
To do that we show that, for any compactly supported function $h$,
the operator $g=\Phi h=\displaystyle \sum_{\alpha \in \ZZ^s}
a^{(k)}(\alpha)h(\cdot-\alpha)$ preserves the regularity of $h$.
Note that, due to $\ra^{(k)}(0) \not =0$, its symbol satisfies
$a^{(k)}_*(z) \not=0$ in the neighborhood of zero. Thus, the
meromorphic function $b^{(k)}_*(z) = 1/a^{(k)}_*(z)$, $z \in \CC^s
\setminus \{0\}$, has the Taylor expansion $b^{(k)}_*(z) =
\displaystyle \sum_{\beta \in \NN_0^s} b^{(k)}(\beta) z^\beta$ in
the neighborhood of zero. Then, due to
$a^{(k)}_*(z)b^{(k)}_*(z)=1$ and by the Cauchy product formula, we
get
$$
 \sum_{\beta \in \NN_0^s} b^{(k)}(\beta) g(\cdot-\beta)= \sum_{\beta \in \NN_0^s} \sum_{\alpha \in \{0,\ldots,N\}^s} b^{(k)}(\beta-\alpha)
 \ra^{(k)}(\alpha) h(\cdot-\beta)=b^{(k)}(0)
 \ra^{(k)}(0) h=h.
$$
Therefore, for the H\"older exponents of $g$ and $h$ we get $\alpha_h \ge \alpha_g$ and, thus, also
$\alpha_{\phi_{k+1}} \ge \alpha_{\phi_{k}}$.
\end{proof}

 \smallskip \noindent For our analysis, for the sequence of masks $\{\mathbf{a}^{(k)},\
k\ge 1\}$ supported on $\{0,\ldots,N\}^s$, we define the so-called
transition matrices $T^{(k)}_{\varepsilon}$, $k \ge 1$,
$\varepsilon \in E \simeq \{0, \ldots,m-1\}^s$, as follows. Accordingly to
\cite{ChenJiaRie}, we set
\begin{equation}\label{def:K}K:=\sum_{r=1}^\infty M^{-r}G,\quad \hbox{where}\quad
G:=\{-m,\ldots,N+1\}^s\,,
\end{equation}
and define the  $|K|\times |K|$ matrices
\begin{equation} \label{def:Tepsilon}
 T^{(k)}_{\varepsilon,\mathbf{a}^{(k)}}:=\left[ \ra^{(k)}
 (\varepsilon+M\alpha-\beta) \right]_{\alpha, \beta \in K}, \quad
 \varepsilon \in E.
\end{equation}
For simplicity we write $T^{(k)}_{\varepsilon}$ instead of
$T^{(k)}_{\varepsilon,\mathbf{a}^{(k)}}$. If the symbols
$\{a^{(k)}_*(z), \ k \ge 1\}$ satisfy sum rules of order $\ell+1$,
then the linear operators $T^{(k)}_{\varepsilon}$ have  common invariant
difference subspaces $V_j \subset \RR^{|K|}$, each of which is orthogonal to
$
 \{ [p(\alpha)]_{\alpha \in K} \in \RR^{|K|}\ : \ p \in \Pi_n\}$, $n=0, \ldots, j$.
The spaces $\Pi_j$ are the spaces of polynomials of total degree
less than or equal to $j=0, \ldots,\ell$. The existence of
$V_j$, $j=0,\ldots,  \ell$, is guaranteed for
$C^\ell-$convergent stationary subdivision schemes and is indeed
used for analysis of convergence and regularity in the stationary
setting. We refer the reader, for example, to the papers
\cite{Cabrelli, CaravettaDahmenMicchelli, Charina, ChenJiaRie,
J95, H03} for more details on the structure of $V_j$ and
for characterizations of regularity of stationary subdivision
schemes in terms of spectral properties of the matrices
$T_{\varepsilon, {\mathbf a}}|_{V_j}$, $\varepsilon \in E$.
Similarly to \eqref{def:Tepsilon}, these matrices are derived from
the stationary mask $\mathbf a$ as follows: define
\begin{equation} \label{def:Tepsilon_a}
 T_{\varepsilon, {\mathbf a}}:=\left[ \ra
 (\varepsilon+M\alpha-\beta) \right]_{\alpha, \beta \in K}, \quad
 \varepsilon \in E,
\end{equation}
and determine their restrictions $T_{\varepsilon, {\mathbf
a}}|_{V_j}$ to $V_j$. Since, in
general, in the non-stationary setting, the existence of such
invariant subspaces is not guaranteed by the regularity of the
limit functions, in this paper we study non-stationary schemes
$\{S_{\ab^{(k)}}, \ k\ge 1\}$ whose sequences of masks possess
sets $\cA$ of limit points, see Definition
\ref{def:set_of_limit_points}. This allows us, similarly to the
stationary setting, to establish a link between the regularity of
non-stationary schemes and the spectral properties of the
collection of square $|K| \times |K|$ matrices $T_{\varepsilon,
\mathbf{a}}$ restricted to $V_j$, $j=0,\ldots,\ell$. This
collection we denote by
$\cT_{\cA}|_{V_j}:=\{T_{\varepsilon,\mathbf{a}}|_{V_j},\
\varepsilon\in E,\ \mathbf{a}\in \cA\}.$

We conclude this section by recalling the notion of the joint
spectral radius of a set of square matrices, see
\cite{RotaStrang}.

\begin{definition}\label{def:JSR} The joint spectral radius (JSR) of a compact collection of square matrices ${\cM}$ is defined by
$
\displaystyle{ \rho({\cM}):=\lim_{n \rightarrow \infty} \max_{M_{1}, \ldots, M_n \in \cM} \left\|\prod_{j=1}^n
 M_{j} \right\|^{1/n}.}$
\end{definition}

\noindent Note that $\rho({\cM})$ is independent of the choice of the matrix norm $\| \cdot \|$.

\section{Convergence and H\"older regularity of non-stationary schemes}\label{sec:convergence_regularity}
In this section we derive sufficient conditions for convergence
and H\"older regularity of a wide class of non-stationary
subdivision schemes. In our proofs we make use of the special
structure of the matrices $T_{\varepsilon, \mathbf{a}}$,
$\mathbf{a} \in \cal A$ in \eqref{def:Tepsilon_a}, and the matrices $T_{\varepsilon}^{(k)}$ in \eqref{def:Tepsilon}
associated with a sequence of masks $\{\mathbf{a}^{(k)},\ \ k \ge
1\}$. This structure is ensured after a suitable change of basis,
which we discuss in \S \ref{subsec:transformation_basis}. In the rest of the
paper, we call such a basis \emph{ a transformation basis}. In
\S \ref{subsec:sum_rules_vs_approx_s_r}, we illustrate two important differences
between sum rules and approximate sum rules.
In \S \ref{subsec:convergence}, we show that a non-stationary
subdivision scheme $\{S_\mathbf{a}^{(k)},\ \ k \ge 1\}$ is
convergent if its symbols satisfy approximate sum rules of order
$1$ and, in addition, its sequence of masks $\{\mathbf{a}^{(k)},\
\ k \ge 1\}$  possesses the set of limit points $\cA$ such that
$\rho\left(\cT_{\cA}|_{V_0}\right)<1$. Note that, after an
appropriate adaptation of the notation in \S \ref{sec:intro}
and \S \ref{sec:background}, the proof of this convergence result is
also valid in the case of a general integer dilation matrix $M$,
the spectral radius of whose inverse satisfies $\rho(M^{-1})<1$.
In \S \ref{subsec:regularity}, we analyze the H\"older
regularity of the basic limit function $\phi_1$ under the
assumptions of approximate sum rules of order $\ell+1$, $\ell \ge
0$, and $\rho\left(\cT_{\cA}|_{V_\ell}\right)<m^{-\ell}$. In
\S \ref{subsec:rapidly}, we prove Theorem
\ref{th:stable_case} and show that under a certain stability
assumption the quantity $\rho\left(\cT_{\cA}|_{V_\ell}\right)$
determines the exact H\"older exponent of the subdivision limits.
This result allows us to prove in \S
 \ref{subsec:convergence_examples} a recent conjecture on regularity of
Daubechies wavelets stated in \cite{DKLR}. In \S
\ref{subsec:convergence_examples},
we also illustrate our results with several examples.

\noindent We start by stating important properties of the set $\cA$.

\begin{proposition}\label{prop:limit-sum-rules} Let $\ell\ge 0.$ Let $\cA$ be
the set of limit points of $\{\mathbf{a}^{(k)},\ \ k \ge 1\}$.
Assume that $\{a_*^{(k)}(z),\ \ k \ge 1\}$ satisfy approximate sum
rules of order $\ell+1$. Then, the symbols associated with the
masks in $\cA$ satisfy sum rules of order $\ell+1$.
\end{proposition}
\begin{proof}
The proof follows from Definition \ref{def:set_of_limit_points}
and the fact that approximate sum rules in Definition
\ref{def:approx_sum_rules} imply that
$\lim_{k\rightarrow\infty}\delta_k=\lim_{k\rightarrow\infty}\mu_k=0$.
\end{proof}

\noindent Next, we would like to remark that the class of
the non-stationary schemes we analyze is not empty.

\begin{remark} In general, for an arbitrary compact set $\cA$ of masks, there exists a
non-stationary subdivision scheme $\{S_{\ab^{(k)}}, \ k \ge 1\}$
with the set of limit points $\cal A$. One possible way of
constructing $\{S_{\ab^{(k)}}, \ k \ge 1\}$ from a given set $\cA$
is presented in Example 2.
\end{remark}

\subsection{Transformation basis} \label{subsec:transformation_basis}
Let $\ell \ge 0$
and $\Pi_j$  be the spaces of polynomials of total degree less
than or equal to $j=0, \ldots,\ell$. If the symbols of the masks
${\mathbf{a}} \in \cA$ satisfy sum rules of order $\ell+1$, then
the corresponding stationary subdivision operators
$S_{\mathbf{a}}$ posses certain polynomial egensequences
$\{p_{\mathbf a }(\alpha), \ \alpha \in\ZZ^s\}$, $p_{\mathbf a}
\in \Pi_j$, $j =0, \ldots,\ell$. These polynomial sequences are
possibly different for different $\mathbf{a}$. For each $j=0,
\ldots, \ell$, the number $d_{j+1}$ of such eigensequences is
equal to the number of monomials $x^\eta$, $\eta \in \NN_0^s$, of
total degree $|\eta|=j$, see \cite{JetterPlonka, JiaJiang}. These
eigensequences, written in a vector form with ordering of the
entries as in \eqref{def:Tepsilon}, become common
left-eigenvectors of the corresponding matrices $T_{\varepsilon,
\mathbf{a}}$. There are at least two different ways of
constructing the so-called \emph{transformation basis} of
$\RR^{|K|}$. The approach in \cite{Cabrelli} makes use of the
eigensequences of the stationary subdivision operator. We cannot
do that as the eigensequences of $S_{\mathbf a}$, $\mathbf{a} \in
\cA$, possibly differ for different $\mathbf{a}$. For that reason,
we follow the approach in \cite{DL1992, P3, P2}, which makes use of the
elements in the common invariant subspaces $V_j$ of
$T_{\varepsilon, \mathbf{a}}$, $\mathbf{a} \in \cA$.  The
\emph{transformation basis} can be constructed as follows: Take
the first unit vector of $\RR^{|K|}$ and extend it to a basis of
$\RR^{|K|}$ by choosing appropriate $d_{j+2}$ vectors from $V_j$,
$j=0, \ldots, \ell-1$, and a complete basis of $V_\ell$. Note
that, any vector from $V_j$ is constructed to be orthogonal to the polynomial vectors
$[p_{\mathbf a}(\alpha)]_{\alpha \in K}$, $p_{\mathbf a} \in \Pi_i$, $i
=0, \ldots, j$. We choose $d_{j+2}$ vectors, say $v_{j,\eta}$,
from $V_j$ in such a way that they are orthogonal to all but
one vector  $[\alpha^\eta]_{\alpha \in K}$ for the corresponding $\eta
\in \NN_0^s$ with $|\eta|=j+1$.

\noindent  This choice of the transformation
basis guarantees that the transformed matrices $T_{\varepsilon,
\mathbf{a}} \in \RR^{|K| \times |K|}$, $\varepsilon \in E$,  are
block-lower triangular and of the form
\begin{equation} \label{eq:T_a_after_transformation}
\left(\begin{array}{ccccc} 1& & & & 0\\ & B_2 & & & \\
b_{1,\varepsilon,\mathbf{a}} & &\ddots& &\\
& b_{2,\varepsilon,\mathbf{a}} & &B_{\ell+1} & \\
&&& b_{\ell+1,\varepsilon,\mathbf{a}}& T_{\varepsilon,
\mathbf{a}}|_{V_\ell}
\end{array} \right),
\end{equation}
where the $d_j \times d_j$ matrices $B_{j}$ are diagonal with
diagonal entries equal to $m^{-j+1}$; the matrices
$b_{j,\varepsilon,\mathbf{a}}$ are of size
$\left(|K|-\displaystyle \sum_{i=1}^j d_i\right) \times d_j$.
Moreover,  if $\{a_*^{(k)}(z),\ k\ge1\}$ satisfy approximate sum
rules of order $\ell+1$, then, after the same change of basis, the
matrices $T_\varepsilon^{(k)} \in \RR^{|K| \times |K|}$,
$\varepsilon \in E$, $k \ge 1$,  are sums of a block-lower and a
block-upper triangular matrices
\begin{equation} \label{eq:T_after_transformation}
\widetilde{T}_\varepsilon^{(k)}+\Delta_\varepsilon^{(k)}:=\left(\begin{array}{ccccc} 1& & & & 0\\ & B_2 & & & \\
b_{1,\varepsilon}^{(k)} & &\ddots& &\\
& b_{2,\varepsilon}^{(k)} & &B_{\ell+1} & \\
&&& b_{\ell+1,\varepsilon}^{(k)}&Q_\varepsilon^{(k)}
\end{array} \right)+ \left(\begin{array}{ccccc} c_{1,\varepsilon}^{(k)}&  \ldots& & &\\
&  c_{2,\varepsilon}^{(k)}& \ldots & & \\
& & \ddots& &\\
&  & &  c_{\ell+1,\varepsilon}^{(k)} & \ldots \\
0&&& & {\rm O}
\end{array} \right),
\end{equation}
where $b_{j,\varepsilon}^{(k)}$  are of size
$\left(|K|-\displaystyle \sum_{i=1}^j d_i\right) \times d_j$;  the
matrices $c_{j,\varepsilon}^{(k)}$ of size $d_j \times
\left(|K|-\displaystyle \sum_{i=1}^{j-1} d_i\right)$; ${\rm O}$ is
the zero matrix of the same size as $Q^{(k)}_\varepsilon$.

\subsection{Sum rules versus approximate sum rules} \label{subsec:sum_rules_vs_approx_s_r}

The following example illustrates two important differences between sum
rules and approximate sum rules stated in Definitions
\ref{def:sumrules} and \ref{def:approx_sum_rules}, respectively.
Firstly, the re-scaling of all symbols
of a non-stationary subdivision masks to ensure that $\mu_k=0$, $k
\ge 1$, can change the properties of the non-stationary scheme if
the sequence $\{\mu_k,\ k\ge 1\}$ is not summable. In other words,
in contrast to the stationary case, the properties of
$a^{(k)}_*(1)$, $k \ge 1$, are crucial for convergence and
regularity analysis of non-stationary schemes. Secondly, even in
the univariate  case, the existence of the factor $(1+z)$ for all
non-stationary symbols $a^{(k)}_*(z)$ and the contractivity of the
corresponding difference schemes do not guarantee the convergence
of the associated non-stationary scheme, if $\{\mu_k, \  k \ge
1\}$ is not summable.

\begin{example}  \label{ex:no_approx_sum_rules}
Let $s=1,\ M=2$. It is well-known that the convergence of
$S_\mathbf{a}$ in the stationary case is equivalent to the fact
that the difference (or derived) scheme $S_\bb$ with the symbol
$b_*(z)$ such that
$$
 {a_*}(z)= \left(1+z \right){b_*}(z),
 \quad z \in \CC\setminus \{0\},
$$
is zero convergent, i.e, for every $v \in \RR^{|K|}$ orthogonal to
a constant vector and $\varepsilon_1, \ldots, \varepsilon_k \in
\{0,1\}$, the norm $\| T_{\eps_1,\mathbf{a}} \cdots T_{\eps_k,
\mathbf{a}} v\|$ goes to zero as $k$ goes to $\infty$. In the
non-stationary case, this characterization is no longer valid.
Consider the non-stationary scheme with the masks
\begin{equation} \label{es:5}
 \mathbf{a}^{(k)}:=\left(1+\frac{1}{k} \right) \mathbf{a}, \quad k \ge 1.
\end{equation}
Note that $\mu_k=\frac{2}{k}$, $\delta_k=0$ and, thus, we can conclude that the
non-stationary scheme $\{S_{\mathbf{a}^{(k)}}, \ k \ge 1\}$ does
not satisfy approximate sum rules. However, it is asymptotically
similar to $S_\mathbf{a}$ and the associated symbols satisfy
$$
   {a_*}^{(k)}(z):= \left(1+z \right)\left(1+\frac{1}{k} \right){b_*}(z),\quad k \ge 1,
 \quad z \in \CC\setminus \{0\}.
$$
We show next that the zero convergence of the associated difference schemes with symbols
$\left(1+\frac{1}{k} \right)b_*(z)$ does not
 imply the convergence of the corresponding non-stationary scheme. Indeed,
 for $\eps_j \in \{0,1\}$, we get
$$
 \|  T_{\eps_1}^{(1)} \cdots T_{\eps_k}^{(k)} v \|=\prod_{j=1}^k
 \left(1+\frac{1}{j}\right) \|  T_{\eps_1,\mathbf{a}} \cdots T_{\eps_k,\mathbf{a}} v  \|=
 (k+1)\| T_{\eps_1,\mathbf{a}} \cdots T_{\eps_k,\mathbf{a}} v \|.
$$
The convergence of $S_\mathbf{a}$ implies the existence of an
operator norm such that
$$
  \|  T_{\eps_1,\mathbf{a}} \cdots T_{\eps_k,\mathbf{a}} v\| \le C \gamma^k, \quad C>0, \quad \gamma <1.
$$
Therefore, the norm $\| T_{\eps_1}^{(1)} \cdots T_{\eps_k}^{(k)}
v \|$ goes to zero as $k$ goes to $\infty$, but the
corresponding non-stationary scheme is not convergent.  Otherwise, the Fourier-transform of
its basic limit function $\phi_1$ would satisfy
$\displaystyle{
 \hat{\phi}_1(\omega)=\prod_{j=1}^\infty {a}_*^{(j)}(e^{-i2 \pi 2^{-j}\omega}), \ \omega \in \RR,}
$
but
$$
 \hat{\phi}_1(0)= \lim_{k \rightarrow \infty}2 \prod_{j=1}^k \left(1+\frac{1}{j} \right){b}_*(1)=\lim_{k \rightarrow \infty}2(k+1){b}_*(1)=\infty.
$$
Note that, if we rescale the masks so that all $\mu_k=0$, $k \ge
1$, we get back the convergent stationary scheme $S_\mathbf{a}$.
\end{example}

\subsection{Convergence}\label{subsec:convergence}
We start by recalling that, in the stationary case, for
convergence analysis via the joint spectral radius approach one
uses the subspace
\begin{equation} \label{def:scalarV}
 V_0:=\{v \in \RR^{|K|}: \ \sum_{j=1}^{|K|} v_j=0\}\,,
\end{equation}
where $K$ is defined in (\ref{def:K}). This subspace also plays an
important role in the proof of the following Theorem
\ref{teo:convergence} that provides sufficient conditions for
convergence of a certain big class of non-stationary schemes. In
the case $M=mI$, $m \ge 2$, Theorem \ref{teo:convergence} is an
instance of Theorem \ref{th:main_Hoelder} with $\ell=0$. Note
though that in the proof of Theorem \ref{teo:convergence} we do
not assume that $M=mI$, $m \ge 2$, and, thus, we need a more
general definition of  approximate sum rules of order $1$.

\begin{definition} \label{def:approx_sum_rulesMorder1} Let $\Xi =\{ e^{2\pi  i  M^{-T} \xi} \ : \   \xi \ \hbox{is a coset representative of} \ \ZZ^s / M^T \ZZ^s\}\,$. The  symbols $\{a^{(k)}_*(z),\ \ k \ge 1\}$ satisfy
approximate sum rules of order $1$, if the sequences $\left\{\mu_k,\ k \ge 1 \right\}$ and $\left\{ \delta_k, \ k\ge 1 \right\}$ with
  \begin{equation}\label{deltak1mu}
 \mu_k:=\left|{a}_*^{(k)}(1)-|\det(M)|\right| \quad \hbox{and}\quad
 \delta_k:=\max_{\epsilon \in \Xi \setminus \{1\}} {| {a}_*^{(k)}(\epsilon)|}
  \end{equation}
are summable.
\end{definition}

In the case of a general dilation matrix, the set $E$ is the set
of coset representatives $E\simeq \ZZ^s / M \ZZ^s$.

\begin{theorem} \label{teo:convergence} Assume that the sequence of symbols
$\{ a_*^{(k)}(z), \ k \ge 1\}$ satisfies approximate sum rules of
order 1 and $\rho\left(\cT_{\cA}|_{V_0}\right)<1$, where $\cA$ is the
set of limit points of $\{ \mathbf{a}^{(k)}, \ k \ge 1\}$. Then the non-stationary scheme
$\{ S_{{\mathbf a}^{(k)}}, \ k \ge 1\}$ is $C^0-$convergent.
\end{theorem}
\begin{proof}
By \cite{GoodmanLee}, the convergence of a non-stationary scheme
is equivalent to the convergence of the associated cascade
algorithm. Thus, to prove the convergence of the non-stationary
scheme $\{ S_{{\mathbf a}^{(k)}}, \ k \ge 1\}$, we show, for $v \in
 \RR^{|K|}$, that the
vector-sequence with the elements
$
 T^{(1)}_{\eps_1} \cdots T^{(k)}_{\eps_k} v, \quad  k \ge 1$,
converges as $k$ goes to infinity for every choice of $\eps_1,
\ldots  \eps_k, \in E$.

\noindent Due to Proposition \ref{prop:limit-sum-rules}, each
$\mathbf{a}\in \cA$ satisfies sum rules of order $1$. Therefore,
by \S \ref{subsec:transformation_basis}, the vector $\left( 1 \ 0
\ldots 0 \right)$ is a common left eigenvector of all matrices
$$
 T_{\eps, \mathbf{a}}= \left( \begin{array}{lcr}
1      & 0\cdots   0 & \\
     &         & \\
b_{\eps,\mathbf{a}} & \quad T_{\eps,\mathbf{a}}|_{V_0} \quad & \\
      &         & \\
\end{array}
\right), \quad \eps \in E, \quad \mathbf{a} \in \fpak.
$$
\noindent Due to the assumption of approximate sum rules of order
$1$, by \S \ref{subsec:transformation_basis}, we have
\begin{equation} \label{eq:decomp}
T^{(k)}_\eps = \widetilde{T}^{(k)}_\eps + \Delta^{(k)}_\eps, \quad
\eps \in E, \quad k \ge 1,
\end{equation}
with
$$
\widetilde{T}^{(k)}_\eps = \left( \begin{array}{lcr}
1      & 0\cdots   0 & \\
     &         & \\
b^{(k)}_{\eps} & \quad Q^{(k)}_{\eps} \quad & \\
      &         & \\
\end{array}
\right), \qquad \Delta^{(k)}_\eps = \left( \begin{array}{ccc}
&  c^{(k)}_\eps              &   \\
         & O                 &       \\
\end{array}
\right).
$$
Thus, the canonical row unit vector $\left( 1 \ 0 \ldots 0
\right)$ is a {\em quasi}-common left-eigenvector of the operators
$T^{(k)}_{\eps}$, $\eps \in E$, i.e.
$
\left( 1 \ 0 \ldots 0 \right)T^{(k)}_{\eps}=\left( 1 \ 0 \ldots 0
\right) + c^{(k)}_\eps$,
where the row vector $c^{(k)}_\eps$ vanishes as $k$ tends to
infinity and the corresponding sequence of norms $\{\|
\Delta^{(k)}_\eps \|, \ k \ge 1\}$ is summable.
Moreover, $b^{(k)}_\eps$ and $Q^{(k)}_\eps$ converge by
subsequences as $k$ goes to infinity to $b_{\eps, \mathbf{a}}$ and
$ T_{\eps,\mathbf{a}}|_{V_0}$ for some $\mathbf{a} \in \cA$,
respectively.

\smallskip \noindent By assumption
$\rho \left( \{ T_{\eps, \mathbf{a}}|_{V_0}, \ \eps \in E,
\mathbf{a} \in \fpak\} \right) < 1$. Thus, the existence of the
 operator norm of $\{ T_{\eps, \mathbf{a}}|_{V_0}, \ \eps
\in E, \mathbf{a} \in \fpak\}$ and the continuity of the joint
spectral radius imply that there exists $\bar k$ such that $\rho
\left( \{ Q^{(k)}_\eps, \eps \in E, k \ge \bar{k}\} \right) < 1$.
This implies that for all vectors $v \in \RR^{|K|}$, the product $
\widetilde{T}^{(1)}_{\eps_1} \cdots \widetilde{T}^{(k)}_{\eps_k}
v$ converges as $k$ goes to infinity for every choice of $\eps_1,\ldots,\eps_k
\in E$.

\noindent By well-known results on the joint spectral radius of
block triangular families of matrices (see e.g. \cite{BerWan}), we
obtain that $ \rho \left( \{ \widetilde{T}^{(k)}_\eps, \ \eps \in
E, k \ge \bar{k}\} \right) = 1.$ Moreover, the family of matrices
$\{ \widetilde{T}^{(k)}_\eps, \ \eps \in E, k \ge \bar{k}\}$ is
non-defective (see e.g. \cite{GuglielmiZennaro}), thus by
\cite{BerWan,RotaStrang}, there exists an  operator norm
$\|\cdot \|$ such that
\begin{equation} \label{eq:exnorm}
\| \widetilde{T}^{(k)}_\eps \| \le 1 \quad \mbox{for all} \quad
\eps \in E, \quad k \ge \bar{k}.
\end{equation}
Due to our assumption that the approximate sum rules of order $1$
are satisfied, we also have
\begin{equation} \label{eq:summ}
\| \Delta^{(k)}_\eps \| \le C \delta_{k} \quad \mbox{where} \quad
\sum\limits_{k=1}^{\infty} \delta_{k}< \infty,
\end{equation}
and $C$ is a constant which does not depend on $k$.

\smallskip \noindent Next, for $n,\ell \in \NN$, we observe that
$$
T^{(n)}_{\eps_n} \cdots T^{(n+\ell)}_{\eps_{n+\ell}} = \left(
\widetilde{T}^{(n)}_{\eps_n} + \Delta^{(n)}_{\eps_n} \right)
\cdots \left( \widetilde{T}^{(n+\ell)}_{\eps_{n+\ell}} +
\Delta^{(n+\ell)}_{\eps_{n+\ell}} \right) =
\widetilde{T}^{(n)}_{\eps_n} \cdots
\widetilde{T}^{(n+\ell)}_{\eps_{n+\ell}} + R_{n,\ell},
$$
where $R_{n,\ell}$ is obtained by expanding all the products. From
(\ref{eq:exnorm})-(\ref{eq:summ}) we get $\lim\limits_{n
\rightarrow \infty} R_{n,\infty} = {\rm O}$ implying convergence
of $\prod\limits_{j=1}^{k} T^{(j)}_{\eps_j} v $ as $k
\rightarrow \infty$. The reasoning for $\lim\limits_{n \rightarrow
\infty} R_{n,\infty} = {\rm O}$ is as follows
$$
 \|R_{n,\infty}\| \le \sum_{j=1}^\infty \left( \sum_{k=n}^{\infty}
 \delta_k
\right)^j=
  \sum_{j=0}^\infty \left( \sum_{k=n}^{\infty} \delta_k \right)^j -1 =
  \displaystyle \sum_{k=n}^{\infty} \delta_k \left(1-\displaystyle
\sum_{k=n}^{\infty} \delta_k \right)^{-1}.
$$
\end{proof}

\begin{corollary} \label{coro:teo_convergence}
Let $\{ S_{{\mathbf a}^{(k)}}, \ k \ge 1\}$ be a $C^0-$convergent
subdivision scheme with the set of limit points $\cA$ such that
$\rho(\cT_\cA|_{V_0})<1$. Then any other asymptotically similar
non-stationary scheme $\{ S_{{\mathbf b}^{(k)}}, \ k \ge 1\}$
satisfying approximate sum rules of order 1 is $C^0-$convergent.
\end{corollary}

\smallskip \noindent
We would like to remark that Theorem \ref{teo:convergence}
generalizes \cite[Theorem 10]{ContiDynManniMazure13} dealing with
the binary univariate case under the assumption that the
non-stationary scheme reproduces constants. Theorem
\ref{teo:convergence} is also a generalization of the
corresponding results in \cite{DynLevin, DynLevinYoon} that
require that stationary and non-stationary schemes are
asymptotically equivalent.

\subsection{$C^\ell-$convergence and H\"older regularity} \label{subsec:regularity}
In this section we prove Theorem \ref{th:main_Hoelder} stated in
the Introduction, i.e. we derive sufficient conditions for
H\"older regularity of non-stationary multivariate subdivision
schemes. Note that Theorem \ref{th:main_Hoelder} with $\ell=0$
also implies the convergence of the corresponding non-stationary
scheme. We, nevertheless,  gave the proof of $C^0$-convergence
separately in Theorem \ref{teo:convergence}, see \S
\ref{subsec:convergence}, to emphasize that it is not affected by
the choice of the dilation matrix $M$, whereas our proof of
$C^\ell$-convergence in this \S does depend on the choice of
$M=mI$, $m \ge 2$. The proof of Theorem \ref{th:main_Hoelder} is long, thus, in
\S \ref{subsec:aux}, we present several crucial auxiliary
results and then prove this theorem  in
\S \ref{subsec:proof_Hoelder}.

\subsubsection{Auxiliary results} \label{subsec:aux}

In the proof of Theorem \ref{th:main_Hoelder} we make use of the
summable sequence $\{\eta_k, \ k \ge 0\}$ which we define next.
Note first that under the assumption
$\rho(\cT_\cA|_{V_\ell})<m^{-\ell}$ of Theorem
\ref{th:main_Hoelder}, there exist $\gamma \in (0,m^{-\ell})$ and
$\bar{k}$ such that
\begin{equation} \label{eq:estimate_Qepsilonk}
 \| Q_\varepsilon^{(k)}\| < \gamma < m^{-\ell}, \quad \varepsilon \in E, \quad k \ge
 \bar{k},
\end{equation}
where $Q_\varepsilon^{(k)}$ are sub-matrices of the matrices
$\widetilde{T}_\varepsilon^{(k)}$ in
\eqref{eq:T_after_transformation}. This property  of
$Q_\varepsilon^{(k)}$ is guaranteed by \cite{BerWan,RotaStrang} and the convergence of
$\{Q_\varepsilon^{(k)}, \ k \ge 1\}$ to $T_{\epsilon,\mathbf{a}}|_{V_j}$. Furthermore, by
approximate sum rules of order $\ell+1$ (Definition
\ref{def:approx_sum_rules}), the sequence $\{\sigma_0:=1, \
\sigma_k=m^{k\ell}\delta_k, \  k \ge 1\}$ is summable and so is
the sequence $\{\eta_k, \ k \ge 0\}$ with
\begin{equation}\label{def:eta_k}
 \eta_k:=\sum_{j=0}^k \sigma_j q^{k-j}, \quad q:=m^{\ell} \gamma.
\end{equation}
Indeed, since $q < 1$, we have
\begin{equation} \label{eq:property_eta_k}
 \sum_{k=0}^{\infty}\eta_k \, = \, \sum_{j=0}^{\infty} \sigma_j
 \sum_{n=0}^{\infty}  q^n \, = \, \frac{1}{1-q}\sum_{j=0}^{\infty}
 \sigma_j \, < \, \infty.
\end{equation}

In the following Lemma \ref{l10} we estimate the asymptotic behavior of the
matrix products
\begin{equation}\label{Pik}
 P_k \: = \ R_1\, \cdots \, R_k,  \quad k \ge 1,
\end{equation}
where, for some non-negative real number $c$, the $(\ell+2) \times (\ell+2)$
matrices $R_j$ are defined by\small
\begin{equation} \label{def:Rj}  R_j:=\ \left(
\begin{array}{cccccc}
1 +\sigma_j m^{-\ell j} & \sigma_j m^{-\ell j} &  \sigma_j m^{-\ell j} & \ldots&
\sigma_j m^{-\ell j} & \sigma_j m^{-\ell j} \\
c & m^{-1} + \sigma_j m^{-(\ell-1)j}& \sigma_j m^{-(\ell-1)j}  & \ldots
&\sigma_j
m^{-(\ell-1)j} & \sigma_j m^{-(\ell-1)j}  \\
c & c &  {}  &  & \vdots & \vdots \\
\vdots & \vdots &  \ddots &  & \vdots & \vdots \\
c & \ldots  &  c &  &  \sigma_j m^{-j} & \sigma_j m^{-j} \\
c & \ldots  &  c & \ddots& m^{-\ell} + \sigma_j &  \sigma_j  \\
c & \ldots  &  c  & \ldots& c &  \gamma
\end{array}
\right).
\end{equation}
\normalsize In particular, in Lemma \ref{l10}, we show that any
norm of the $r-$th column of the matrix product~$P_k$ is bounded
uniformly over $k \ge 1$, i.e. that any norm of the column $P_k
e_r$, with $e_r$ being the standard $r$-th unit vector, is bounded
uniformly over $k \ge 1$.
 .

\begin{lemma}\label{l10} For every $k \in \mathbb{N}$
$$
\|P_ke_r\|_1=\left\{
               \begin{array}{ll}
                 {\cal O}(m^{-(r-1)k}), & r = 1, \ldots , \ell+1, \\
                 {\cal O}( m^{(\ell+1)}\, m^{-\ell k}), & r=\ell+2.
               \end{array}
             \right.
$$
\end{lemma}
\begin{proof} For simplicity of presentation we consider the  case of  $M=2I$,
i.e $m=2$. Let $C_1$ be the smallest constant such that for each~$r = 1, \ldots
, \ell+1$,  the $r-$th column of $R_1$
 does not exceed~$C_1 2^{-(r-1)}$, and the $(\ell+2)$nd column does not exceed~$C_1 2 \eta_1$. We show by induction on~$k$ that the
 sum of the $r-$th column entries of $P_k$ does not exceed $C_k\, 2^{-(r-1)k}$, $r = 1, \ldots , \ell+1$, or $C_k\, 2^{(\ell+1)}\, 2^{-\ell k}\, \eta_k$, $r=\ell+2$, where
 $$
  C_k \ = \ C_1 \prod_{j=2}^{k} \Bigl(1 + 2^{\, \ell+1} \sigma_j \, + \, c \, 2^{\, \ell} 2^{\,2-j} + c 2^{\, 2\ell+1}\, \, \eta_{j-1}\Bigr)\, ,
 \qquad k \ge 2\,.
 $$
 Due to
 \begin{equation}\label{ck}
 C_{k} \ =  \ C_{k-1}\Bigl(1 + 2^{\, \ell+1} \sigma_k \, + \, c \, 2^{\, \ell} 2^{\,2-k} + c 2^{\, 2\ell+1}\, \, \eta_{k-1}\Bigr)\, ,
 \quad k \ge 2\, ,
 \end{equation}
 the sequence $\{C_k, \ k \ge 1\}$ increases and converges to
\begin{equation}\label{c}
 \tilde C \ := \ C_1\, \prod_{j=2}^{\infty} \, \Bigl(1 + 2^{\, \ell+1} \sigma_j \, + \, c \, 2^{\, \ell} 2^{\,2-j} +
 c 2^{\, 2\ell+1}\, \, \eta_{j-1}\Bigr).
\end{equation}
Since the sums $\sum\limits_{j=1}^{\infty} \sigma_j \, , \,
\sum\limits_{j=1}^{\infty} 2^{\, 2 - j} $, and $\sum\limits_{j=1}^{\infty}
\eta_{j-1}$ are all finite, the infinite product in \eqref{c} converges.
 By induction assumption, we have $\| P_{k-1} e_j\bigr\|_1 \le C_{k-1}2^{-(j-1)(k-1)}$, for $j=1, \ldots , \ell+1$,
 and $\| P_{k-1} e_{\ell+2}\bigr\|_1 \le C_{k-1}2^{\ell+1}\, 2^{-\ell (k-1)}\eta_{k-1}$.
 Since the $r-$th column of~$P_k$ is $P_ke_r$, where $e_r$ is the $r-$th basis vector of~$\mathbb{R}^{\ell+2}$,  we have
 $$
 \Bigl\| P_ke_r \Bigr\|_{1}  \ = \ \Bigl\| P_{k-1} R_k e_r \Bigr\|_{1} \ \le \
 \sum_{j=1}^{\ell+2} \, \Bigl\| P_{k-1} e_j\Bigr\|_1  \bigl( R_k\bigr)_{jr}\, .
 $$
 Thus,
 \begin{equation}\label{induc}
 \bigl\| P_{k}e_r \bigr\|_1 \ \le \  \sum_{j=1}^{\ell+2} \, \Bigl\| P_{k-1} e_j\Bigr\|_1  \bigl( R_k\bigr)_{jr}\, .
 \end{equation}
 Next we consider the following three cases.
\smallskip

 Case 1: $r=1$.  The first column of the matrix~$R_k$ is $(1+ \sigma_k 2^{-\ell k}, c, \ldots , c)^T$.  By induction assumption and due to $\displaystyle  \sum_{j=2}^{\ell+1}  2^{-(j-1)(k-1)}=0$ for $\ell=0$ and \eqref{ck}, the
 estimate \eqref{induc} yields
 \begin{eqnarray*}
 \bigl\| P_{k}e_1 \bigr\|_1 \ &\le& \ C_{k-1}\, \Bigl(  1 + \sigma_k 2^{-\ell k} + \,  c \sum_{j=2}^{\ell+1}  2^{-(j-1)(k-1)}\, + \,
 c\, 2^{\ell+1}\, 2^{-\ell (k-1)}\eta_{k-1} \Bigr) \\&\le&
  C_{k-1}\, \Bigl(  1 + \sigma_k 2^{-\ell k} + \,  2\, c\, 2^{-(k-1)}\, + \,
 c\, 2^{\ell+1}\, \eta_{k-1} \Bigr) \ \le \ C_k.
 \end{eqnarray*}

 Case 2: $ 2 \le r\le \ell+1$. The $r-$th column of the matrix~$R_k$ is
  $$
  (R_k)_r \ = \
  \bigl(\sigma_k2^{-\ell k}, \ldots , \sigma_k 2^{-(\ell-(r-2))\, k}, 2^{-(r-1)}+ \sigma_k 2^{-(\ell-(r-1))\, k}, c, \ldots , c \bigr)^T\, .
  $$
 By induction assumption the estimate in~(\ref{induc}) becomes
 \begin{eqnarray*}
 \bigl\| P_{k}e_r \bigr\|_1 \ &\le& \ C_{k-1}\, \Bigl( \sigma_k \sum_{j=1}^{r-1} 2^{-(\ell+1-j)k}2^{-(j-1)(k-1)} \, + \,
 \bigl( 2^{-(r-1)} + \sigma_k 2^{-(\ell+1 - r)k}\bigr) 2^{-(r-1)(k-1)} \\
 &+& \, c \sum_{j=r+1}^{\ell+1} 2^{-(j-1)(k-1)} \, + \, c \, 2^{\, \ell+1} 2^{-\ell (k-1)}\, \, \eta_{k-1}\, \Bigr)\\
 &\le&
C_{k-1}\, \Bigl( \sigma_k 2^{-\ell k}\sum_{j=1}^{r} 2^{j-1} \, + \, 2^{-(r-1)k}
\, + \, 2\, c  \, 2^{-r(k-1)}\, + \, c\, 2^{\, \ell+1} 2^{-(r-1)(k-1)} \, \,
\eta_{k-1}\, \Bigr)\\ &\le&
 C_{k-1}\, \Bigl( \sigma_k 2^{-(r-1)k} 2^{r}  \, + \, 2^{-(r-1)k} \, + \, c \, 2^{r+1}\, 2^{-(r-1)k} 2^{-k}\, + \,
c\, 2^{-(r-1)k} 2^{\, \ell+r}\, \, \eta_{k-1}\, \Bigr)\\ &\le&
 C_{k-1}\, 2^{-(r-1)k}\Bigl(1 + 2^{\, \ell+1} \eta_k \, + \, c \, 2^{\, \ell} 2^{\,2-k} + c\, 2^{\, 2\ell+1}\, \, \eta_{k-1}\Bigr) \le \ C_{k}\, 2^{-(r-1)k}\, .
 \end{eqnarray*}
 \smallskip

 Case 3: $ r = \ell+2$. The last column of the matrix~$R_k$ is
  $\bigl(\sigma_k2^{-\ell k}, \ldots , \sigma_k2^{-k}, \sigma_k , \gamma \bigr)^T$.
  Note that by definition of~$\eta_k$ we have~$\eta_k - \sigma_k = \eta_{k-1}$, and recall that
  $\gamma < 2^{-\ell}$.
 Then, by induction assumption, we get
 \begin{eqnarray*}
 \bigl\|P_k e_{\ell+2} \bigr\|_1 \, &\le& \, C_{k-1}\, \Bigl( \sigma_k \sum_{j=1}^{\ell+1} 2^{-(\ell+1-j)k}2^{-(j-1)(k-1)} \, + \,
2^{\, \ell+1}2^{-\ell (k-1)}\, \gamma \, \eta_{k-1}
 \Bigr)\\ &=&
 C_{k-1}\, \Bigl( \sigma_k 2^{-\ell k}\sum_{j=1}^{\ell+1} 2^{j-1} \, + \, 2^{\, \ell+1} 2^{-\ell k} \, \bigl(\eta_k - \sigma_k \bigr)\, \Bigr)\\  &\le&
C_{k-1}\, \Bigl( \sigma_k 2^{-\ell k} 2^{\, \ell +1}\, + \, 2^{\, \ell +1} 2^{-\ell k} \, \eta_k\, - \, 2^{\, \ell+1}\, 2^{-\ell k} \sigma_k \Bigr)\\
 &=& \, C_{k-1} 2^{\, \ell +1} 2^{-\ell k} \, \eta_k\, \le \,
C_k 2^{\, \ell+1} 2^{-\ell k} \, \eta_k\, .
\end{eqnarray*}
\end{proof}

\noindent The estimates in Lemma \ref{l10} allow us to estimate the norms of the columns
of the matrix products $T_{\varepsilon_1}^{(1)} \cdots T_{\varepsilon_k}^{(k)}$,
$\varepsilon_1, \ldots, \varepsilon_k \in E$.

\begin{lemma} \label{lemma:estimate_products_of_T}
Let $\varepsilon_1, \ldots, \varepsilon_k \in E$, $\ell\ge 0$. Assume that the
symbols of $\{S_{{\mathbf a}^{(k)}}, \ k \ge 1\}$ satisfy approximate sum rules
of order $\ell+1$ and $\rho(\cT_\cA|_{V_\ell})< m^{-\ell}$. Then the norms of
the columns of $T_{\varepsilon_1}^{(1)} \cdots T_{\varepsilon_k}^{(k)}$ with
indices $\displaystyle 1+\sum_{j=1}^{r-1} d_j, \ldots, \sum_{j=1}^{r} d_j$ are
equal to ${\cal O}(m^{-(r-1)k})$ for $r=1, \ldots,  \ell+1$. The norms of the
other columns of this matrix product are equal to ${\cal O}(m^{-\ell k}
\eta_k)$.
\end{lemma}
\begin{proof}  Let $\eps \in E$. Under the assumptions of Theorem
\ref{th:main_Hoelder}, the matrices
$\widetilde{T}_\varepsilon^{(k)}$ in
\eqref{eq:T_after_transformation} have the following properties:
the matrix sequences $\{b_{j,\varepsilon}^{(k)}, \ k \ge 1\}$ and
$\{Q_{\varepsilon}^{(k)}, \ k \ge 1\}$ converge by subsequences as
$k$ goes to $\infty$, respectively, to
$b_{j,\varepsilon,\mathbf{a}}$ and
$T_{\varepsilon,\mathbf{a}}|_{V_\ell}$ for some $\mathbf{a} \in
\cal A$; there exists $c>0$ such that all the norms
$\|b_{j,\varepsilon,\mathbf{a}}\|_\infty \le c < \infty $; the
estimate in \eqref{eq:estimate_Qepsilonk} holds for $0 < \gamma <
m^{-\ell}$ and for some matrix norm $\|\cdot\|_{ext}$.
Furthermore, approximate sum rules of order $\ell+1$ and the
definition of $\sigma_k$ imply that the entries of the matrices
$c_{j,\varepsilon}^{(k)}$, $j=1, \ldots, \ell+1$, are bounded by
$\sigma_k m^{-(\ell+1-j)k}$. Next, let $L_0=0$ and $L_i=\displaystyle \sum_{j=1}^{i} d_j$, $i=1,
\ldots, \ell+1$, with $d_j$ defined in \S
\ref{subsec:transformation_basis}. Set $L=L_{\ell+1}$ and write a
vector $v=(v_1, \ldots, v_{|K|})^T \in \RR^{|K|}$ as
$$
 v=\left( v^{[1]}, v^{[2]}, \ldots,  v^{[\ell+1]}, v^{[\ell+2]} \right)
$$
with
$$
v^{[i]}:=(v_{L_{i-1}+1}, \ldots, v_{L_{i}})^T,\ i=1,\ldots,\ell+1,\quad v^{[\ell+2]}:=(v_{L+1},\ldots,  v_{|K|})^T.
$$
Consider the vector norm $\displaystyle{
 \|v\|:=\sum_{i=1}^{\ell+1} \| v^{[i]}\|_\infty+ \|v^{[\ell+2]}\|_{ext}, \quad v \in \RR^{|K|}.}
$
Then
$$
 \|T_\varepsilon^{(k)} v\| \le \|R_k \tilde{v}\|, \quad
 \tilde{v}=(\| v^{[1]}\|_\infty, \ldots, \| v^{[\ell+1]}\|_\infty,
 \|v^{[\ell+2]}\|_{ext}) \in \RR^{\ell+2},
$$
where $R_k$ is given in \eqref{def:Rj}. Analogously, we get
$$
 \|T^{(1)}_{\varepsilon_1} \cdots T^{(k)}_{\varepsilon_k} v\| \le \|R_1 \cdots R_k \tilde{v}\|.
$$
The claim follows by Lemma \ref{l10}.
\end{proof}

\subsubsection{Proof of Theorem \ref{th:main_Hoelder}} \label{subsec:proof_Hoelder}

The proof of  Theorem \ref{th:main_Hoelder} is long, so we split it into two
parts: Proposition~\ref{prop:convergence_Cl} and
Proposition~\ref{prop:Hoelder_proof}. In the first part of the proof, given in
Proposition \ref{prop:convergence_Cl}, we show that the assumptions of
Theorem~\ref{th:main_Hoelder} are indeed sufficient for the $C^\ell-$convergence
of non-stationary schemes. In particular, we let $f \in C^\ell(\RR^s)$ be
compactly supported, stable and refinable with respect to the dilation matrix
$M=mI$ and the mask $\mathbf{e} \in \ell_0(\ZZ^s)$. Then, for every $j=0,
\ldots,\ell$ and for every $\nu \in \NN_0^s$, $|\nu|=j$, we consider the
sequence $\{D^\nu f_k, \ k \ge 1\}$, where for $f_k:= \cT^{(1)} \ldots
 \cT^{(k)}f$
\begin{equation} \label{def:D_nu_f_k}
 D^\nu f_k=m^{jk} \cT^{(1)} \ldots \cT^{(k)} D^\nu f, \quad \cT^{(k)} f=\displaystyle \sum_{\alpha \in \ZZ^s}
\ra^{(k)}(\alpha) f(M\cdot-\alpha),
\end{equation}
i.e. $\cT^{(k)}$ is the transition operator associated with the mask ${\bf
a}^{(k)}$, and show that $\{D^\nu f_k, \ k \ge 1\}$ converges uniformly to the
$\nu-th$ partial derivative of $\phi_1$.

\begin{proposition} \label{prop:convergence_Cl} Let $\ell \ge 0$.  Assume that
the symbols of $\{ S_{{\mathbf a}^{(k)}}, \ k \ge 1\}$ satisfy approximate sum
rules of order $\ell+1$ and $\rho(\cT_\cA|_{V_\ell})<m^{-\ell}$. Then, for every
$j=0, \ldots,\ell$ and for every $\nu \in \NN_0^s$, $|\nu|=j$, the sequence
$\{D^\nu f_k, \ k \ge 1\}$ in \eqref{def:D_nu_f_k} converges uniformly to the
$\nu-th$ partial derivative of $\phi_1$. Moreover, there exists a constant $C>0$
independent of $k$ such that for $\eta_n$ as in (\ref{def:eta_k}) we have
\begin{equation}\label{eq:der}
 \|D^\nu f_k-D^\nu \phi_1\|_{\infty} \le C \sum_{n=k}^\infty \eta_n, \quad |\nu|=\ell, \quad k \ge 1.
\end{equation}
\end{proposition}

\begin{proof} Note that, by \cite[p.137]{ChenJiaRie}, the
function $f$ in \eqref{def:D_nu_f_k} is an appropriate starting function for the cascade
algorithm. Moreover, by  \cite[Theorem 6.3]{J98}, the assumptions
on $f$ imply that $f$ satisfies Strang-Fix conditions of order
$\ell+1$, i.e. its Fourier transform $\hat{f}$ satisfies
$$
 \hat{f}(0)=1, \quad D^\mu \hat{f}(\alpha)=0, \quad \alpha \in \ZZ^s \setminus \{0\}, \quad \mu \in \NN_0^s, \quad |\mu|< \ell+1.
$$
Consequently, its derivatives $D^\nu f$, $\nu \in \NN_0^s$,
$|\nu|=j$, $j=1, \ldots,\ell$, satisfy
$$
  D^\mu \widehat{( D^\nu f)}(\alpha)=0, \quad \alpha \in \ZZ^s, \quad \mu \in \NN_0^s, \quad |\mu|<
  j.
$$
Thus, by the Poisson summation formula, we get
\begin{equation} \label{prop2:propertiesDerivatives}
 \sum_{\alpha \in \ZZ^s} p(\alpha) D^\nu f(x-\alpha)=0, \quad x \in \RR^s,
\end{equation}
for all polynomial sequences $\{p(\alpha), \ \alpha \in \ZZ^s\}$,
$p \in \Pi_j$. Note that we can chose $f$ such that $\hbox{supp} f
\cap \ZZ^s \subset K$. Then, the properties
\eqref{prop2:propertiesDerivatives} of $D^\nu f$ imply that, after
the transformation  discussed in \S
\ref{subsec:transformation_basis}, the first $\displaystyle
\sum_{i=1}^{j} d_i$ entries of the vectors
$$
 v(x):=\left( D^\nu f(x+\alpha)\right)_{\alpha \in K}, \quad x \in
 [0,1]^s, \quad |\nu|=j, \quad j=1, \ldots,\ell,
$$
are equal to zero.
Note that the ordering of the  entries in $ v(x)$ corresponds to
the ordering of the columns of $T_\varepsilon^{(k)}$ defined in
\eqref{def:Tepsilon}. By Theorem \ref{teo:convergence}, the limit
functions of the non-stationary scheme are $C^0(\RR^s)$, i.e. the
sequence $\{f_k, \  k\ge 1\}$ is a uniformly convergent Cauchy
sequence. Similarly to the stationary case, to show that the
non-stationary scheme is $C^j-$convergent, $j=1,\ldots,\ell$, we
need to study the uniform convergence of the sequences $\{D^\nu
f_k, \ k\ge 1\}$ for all $\nu \in \NN_0^s$, $|\nu|=j$.
Equivalently,  for every choice of $\varepsilon_1, \ldots \varepsilon_k
\in E$, need to study the convergence of the vector-sequences
$\{m^{j k} T^{(1)}_{\varepsilon_1} \cdots T^{(k)}_{\varepsilon_k}
w, \ k \ge 1\}$, where $T_\varepsilon^{(k)}$ are defined from
$\{{\mathbf a}^{(k)}, \ k \ge 1\}$  and the vector $w \in
\RR^{|K|}$ is arbitrary and such that its first $\displaystyle \sum_{i=1}^{j}
d_i$ entries are zero. Lemma \ref{lemma:estimate_products_of_T},
the structure of $w$ and the summability of $\{\eta_k, \ k \ge
1\}$ imply the convergence of the vector-sequences $\{m^{j k}
T^{(1)}_{\varepsilon_1} \cdots T^{(k)}_{\varepsilon_k} w, \ k \ge
1\}$ for $j=1, \ldots, \ell$. Thus,  the non-stationary scheme is
$C^\ell-$convergent.

\smallskip \noindent We prove next the estimate \eqref{eq:der}. Let $\nu \in \NN_0^s$,
$|\nu|=\ell$. Due to $\displaystyle \phi_1=\lim_{k \rightarrow
\infty} \cT^{(1)} \ldots \cT^{(k)} f$ and by the assumption of
refinability of $f$, i.e. $\displaystyle f=\cT f=\sum_{\alpha \in
\ZZ^s} \re(\alpha)f(Mx-\alpha)$, we have
\begin{eqnarray*}
 \|D^\nu f_k-D^\nu \phi_1\|_{\infty}&\le& \sum_{n=k}^\infty
 \|D^\nu f_{n+1}- D^\nu f_{n} \|_{\infty} \\ &=& \sum_{n=k}^\infty
 m^{\ell (n+1)} \| \cT^{(1)} \ldots \cT^{(n)} \left(\cT^{(n+1)}-\cT \right) (D^\nu f)(M^{-(n+1)} \cdot)\|_{\infty}.
\end{eqnarray*}
As above, to estimate the norms $\| \cT^{(1)} \ldots \cT^{(n)}
\left(\cT^{(n+1)}-\cT \right) (D^\nu f)(M^{-(n+1)}
\cdot)\|_{\infty}$, we need to estimate the vector-norms of
$$
 T_{\varepsilon_1}^{(1)} \cdots T^{(n)}_{\varepsilon_{n}}
 \left(T_{\varepsilon_{n+1}}^{(n+1)}-T_{\varepsilon_{n+1},
 \mathbf{e}} \right)w,
$$
where $|K| \times |K|$ matrices $T_{\varepsilon,\mathbf{e}}$,
$\varepsilon \in E$, are derived from the mask $\mathbf{e}$, see
\eqref{eq:T_a_after_transformation}, and the first $\displaystyle
\sum_{j=1}^{\ell+1} d_j$ entries of the  vector $w \in \RR^{|K|}$
are zero. By assumption, there exists a constant $\beta>0$ such
that the entries of all $b_{j,\varepsilon,\mathbf{e}}$ and
$b^{(k)}_{j,\varepsilon}$ are less than $\beta$ in the absolute
value. The approximate sum rules of order $\ell+1$ imply that the absolute values of
the entries of the vectors
$\left(T_{\varepsilon_n}^{(n+1)}-T_{\varepsilon_n,\mathbf{e}}
\right) w$ with indices $\displaystyle 1+\sum_{j=1}^r d_j, \ldots,
\sum_{j=1}^{r+1} d_j$, $r=0, \ldots, \ell$, are bounded
respectively by $\sigma_{n+1}m^{-(\ell-r)(n+1)}$. All other
entries are bounded by $2\beta$. Thus, by Lemma
\ref{lemma:estimate_products_of_T}, we get that the entries  of
the vectors $m^{\ell (n+1)} T_{\varepsilon_1}^{(1)} \cdots
T^{(n)}_{\varepsilon_{n}}
\left(T_{\varepsilon_{n+1}}^{(n+1)}-T_{\varepsilon_{n+1},
\mathbf{e}} \right)w$ with indices $\displaystyle 1+\sum_{j=1}^r
d_j, \ldots, \sum_{j=1}^{r+1} d_j$, $r=0, \ldots, \ell$,  are
equal to ${\cal O} (\sigma_{n+1})$, all other entries are equal to
${\cal O}(\eta_n)$. Therefore, by definition of $\{\eta_k, \ k \ge
1\}$ in \eqref{def:eta_k}, we get
$$
 \|D^\nu f_k- D^\nu \phi_1\|_\infty \le C \sum_{n=k}^\infty \eta_n, \quad k \ge
 1,
$$
for some $C>0$ independent of $k$.
\end{proof}

The second part of the proof of Theorem \ref{th:main_Hoelder} is given in
Proposition \ref{prop:Hoelder_proof} which yields the desired estimate for the
H\"older regularity $\alpha$ of the scheme $\{ S_{{\mathbf a}^{(k)}}, \ k \ge
1\}$.

\begin{proposition}\label{prop:Hoelder_proof} Let $k \ge 1$, $h \in \RR^s$,  $m^{-(k+1)}< \|h\|_\infty \le
m^{-k}$ and  $\ell \ge 0$.  Assume that the symbols of $\{ S_{{\mathbf
a}^{(k)}}, \ k \ge 1\}$ satisfy approximate sum rules of order $\ell+1$ and
$\rho(\cT_\cA|_{V_\ell})<m^{-\ell}$. Then there exists a constant $C>0$
independent of $k$ such that, for $\eta_n$ as in (\ref{def:eta_k}), we have
\begin{equation} \label{eq:estimate_in_Hoelder_proof}
 \|D^\nu \phi_1(\cdot+h)-D^\nu \phi_1(\cdot)\|_\infty \le C
 \sum_{n=k}^\infty \eta_n, \quad \nu \in \NN_0^s, \quad
 |\nu|=\ell.
\end{equation}
Moreover, the H\"older exponent $\alpha$ of $\phi_1 \in C^\ell(\RR^s)$ satisfies
\begin{equation}\label{estimate_Hoelder_reg}
 \alpha \ge \min \left\{ -\log_m \rho_\cA, -\limsup_{k \rightarrow \infty} \frac{\log_m
 \delta_k}{k}\right\}.
\end{equation}
\end{proposition}
\begin{proof} Let $k \ge 1$, $|\nu|=\ell$, and $h \in \RR^s$ satisfy
$m^{-(k+1)}< \|h\|_\infty \le m^{-k}$.  To derive the estimate in
\eqref{eq:estimate_in_Hoelder_proof}, we use the triangle
inequality
\begin{eqnarray} \label{prop3:estimate}
 \|D^\nu \phi_1(\cdot+h)&-&D^\mu \phi_1(\cdot)\|_\infty \le
 \|D^\mu \phi_1(\cdot+h)-D^\mu f_k(\cdot+h)\|_\infty \\&+&
 \|D^\mu \phi_1-D^\mu f_k\|_\infty +
 \|D^\mu f_k(\cdot+h)-D^\mu f_k(\cdot)\|_\infty, \notag
\end{eqnarray}
where $\{f_k, \ k \ge 1\}$ are defined in \eqref{def:D_nu_f_k}, and estimate
each of the summands on the right hand side. Note that, for $\Delta_h
f_k:=f_k(\cdot+h)-f_k(\cdot)$, we have
$$
 \Delta_h D^\mu f_k =m^{\ell k}\cT^{(1)} \ldots \cT^{(k)} \Delta_{m^k
 h} D^\mu f.
$$
Due to $\|m^k h\|_\infty \le 1$ and by the definition of
$\Delta_h$, we have
$$
 \hbox{supp} \Delta_{m^k h}D^\nu f \subset \hbox{supp}
 f+[-1,1]^s,
$$
where without loss of generality we assume that $\left(\hbox{supp}
 f+[-1,1]^s \right) \cap \ZZ^s \subseteq K$. Define the vector-valued function
$$
 v(x):=(\Delta_{m^k h} D^\nu f(x+\alpha))_{\alpha \in K}, \quad
 x \in [0,1]^s.
$$
By the same argument as in the proof of Proposition
\ref{prop:convergence_Cl} and by the definition of the operator
$\Delta_h$, the first $\displaystyle{\sum_{j=1}^{\ell+1}} d_j$
components of $v$ are zero for all $x \in \RR^s$. Therefore, by
Lemma \ref{lemma:estimate_products_of_T}, we get
$$
 \|T_{\varepsilon_1}^{(1)} \cdots T_{\varepsilon_k}^{(k)}
 v(x)\|
 \le C_1 m^{-\ell k} \eta_k \| v(x)\| \le C_1 m^{-\ell k}\eta_k \,  2 \,
 C_2
 \, |K|, \quad x \in \RR^s,
$$
where  $\|v(x)\| \le 2 \, C_2 \, |K|$, $x \in \RR^s$, due to
$\max_{|\nu|=\ell}\|D^\nu f\|_\infty \le C_2$. Thus,
$$
 \| D^\nu f_k(\cdot+h)-D^\nu f_k(\cdot)\|_\infty \le C_3 \eta_k,
 \quad C_3:=C_1 m^{-\ell k} \, 2 \, C_2 \, |K|.
$$
The estimates for the two remaining terms in
\eqref{prop3:estimate} and, thus, the estimate
\eqref{eq:estimate_in_Hoelder_proof} follow  by \eqref{eq:der}. Next, we derive the lower bound for the H\"older exponent $\alpha$
of $\phi_1$. Note that, by definition of $\sigma_k$, we have the
equivalence
$$
\lim \sup_{k \rightarrow \infty} \sigma_k^{1/k} \ge 1 \quad
\Leftrightarrow \quad \ell + \lim \sup_{k \rightarrow \infty}
\frac{\log_m\delta_k}{k} \ge 0.
$$
Thus, if $\lim \sup_{k \rightarrow \infty} \sigma_k^{1/k} \ge 1$,
then
$
 \min\left\{-\log_m \rho_\cA, -\limsup \frac{\log_m \delta_k}{k}
 \right\}\le \ell
$
and the estimate \eqref{estimate_Hoelder_reg} holds, since $\phi_1
\in C^\ell(\RR^s)$ and, thus, $\alpha \ge \ell$.
 Otherwise, if
$\displaystyle \lim \sup_{k \rightarrow \infty} \sigma_k^{1/k} <
1$, then there exists $\theta$ such that
$\displaystyle
 \lim \sup_{k\rightarrow \infty} \sigma_k^{1/k}< \theta <1
$
and, thus, a constant $C_0>0$ such that $\sigma_k \le C_0
\theta^k$, $k \ge 1$. Therefore, by definition of $\eta_k$ and
using the estimate \eqref{eq:estimate_in_Hoelder_proof}, we get
\begin{eqnarray*}
 \|\Delta_h D^\nu \phi_1 \|_\infty &\le& C\left(\eta_k+\sum_{n=k+1}^{\infty} \eta_n \right)
 =C \left( \sum_{j=0}^k \sigma_j q^{k-j}+\frac{1}{1-q}
 \sum_{n=k+1}^\infty \sigma_n \right) \\&\le& C \, C_0\left( \sum_{j=0}^k
 \theta^j q^{k-j}+\frac{1}{1-q}\sum_{n=k+1}^\infty \theta^n
 \right)\\
 &\le& C \, C_0\left( \sum_{j=0}^k \max\{\theta,q\}^j \max\{\theta,q\}^{k-j}+
 \frac{\theta^{k+1}}{1-q}\sum_{n=k+1}^\infty \theta^{n-(k+1)} \right) \\
 &\le& C \, C_0\left((k+1) \max\{\theta,q\}^k+
 \frac{\theta^{k+1}}{(1-q)(1-\theta)}\right).
\end{eqnarray*}
Therefore, due to $0 \le 1-q<1$, we get
\begin{eqnarray*}
 \|\Delta_h D^\nu \phi_1 \|_\infty
 \le C_4 (k+1) \max\{\theta,q\}^k, \quad C_4:=\max \left\{1, \frac{\theta}{1-\theta}
 \right\}\, \frac{C\, C_0}{1-q}.
\end{eqnarray*}
Moreover, due to $\|h\|_\infty \le m^{-k}$,  we
have
$$
 \max\{\theta,q\}^k=m^{k \, \text{log}_m \max\{\theta,q \}} \le \|h\|_\infty^{-\text{log}_m
 \max\{\theta,q\}}
$$
and, from  $\frac{1}{m}\|h\|_\infty \le
m^{-(k+1)}$, we get $(k+1) \le \log_m \frac{m}{\|h\|_\infty}$.
Thus,
$$
 \|\Delta_h D^\nu \phi_1\|_\infty \le C_1 \log_m \left(\frac{m}{\|h\|_\infty} \right)\,
 \|h\|_\infty^{-\log_m
 \max\{\theta,q\}}.
$$
Note that for any $\epsilon \in (0,1)$, due to the fact that $-\log(t)$ is bounded by
$t^{-\epsilon}$ for sufficiently small $t$, we get, for small $\|h\|_\infty$,
$$
 \|\Delta_h D^\nu \phi_1\|_\infty \le C_1
\|h\|_\infty^{-\log_m
\max\{\theta,q\}-\epsilon}.
$$
By \eqref{eq:estimate_Qepsilonk} and \eqref{def:eta_k},
we have $q>m^\ell \rho_\cA$. Thus, since $\theta >\lim
\sup_{k\rightarrow \infty} \sigma_k^{1/k}$, we get
\begin{eqnarray*}
 \alpha \ge \ell-\log_m
 \max\{\theta,q\} &=&
 \ell-\max\{\ell+\log_m \rho_\cA, \ell+\lim \sup_{k\rightarrow
\infty} \frac{\log_m\delta_k}{k}\} \\&=&\min \{-\log_m \rho_\cA,
-\lim \sup_{k\rightarrow \infty} \frac{\log_m\delta_k}{k}\}.
\end{eqnarray*}
\end{proof}.

\noindent
Combining Propositions \ref{prop:convergence_Cl} and
\ref{prop:Hoelder_proof}, we complete the proof of Theorem
\ref{th:main_Hoelder}.

\subsection{Rapidly vanishing approximate sum rules defects}
\label{subsec:rapidly}

The following immediate consequence of Theorem
\ref{th:main_Hoelder} states that, if the sequence of defects
$\{\delta_k, \ k \ge 1\}$ of the approximate sum rules decays
fast, then the lower bound on the H\"older exponent $\alpha$ of
$\phi_1$ only depends on the joint spectral radius $\rho_\cA$ of
the set $\cT_\cA|_{V_\ell}$.

\begin{corollary} \label{c10}
Assume that the symbols of $\{ S_{{\mathbf a}^{(k)}}, \ k \ge 1\}$ satisfy approximate sum
rules of order $\ell+1$ and $\rho(\cT_\cA|_{V_\ell})<m^{-\ell}$. If
$\limsup\limits_{k \to \infty}  \delta_k^{1/k} \, < \, \rho_\cA$,
then $\alpha \,  \ge \, - \log_m \rho_\cA$.
\end{corollary}

\noindent Next, in this subsection we prove Theorem \ref{th:stable_case} stated in the Introduction. It
shows that the inequality $\alpha \,  \ge \, - \log_m \rho_\cA$ in
Corollary \ref{c10} becomes equality, if the set $\cA$ of the
limit points of the sequence $\{\mathbf{a}^{(k)}, \ k\ge 1 \}$
consists only of a single element $\mathbf{a}$ and the
corresponding refinable limit function  of $S_\ab$ is stable. Note that  Theorem \ref{th:stable_case} is a
generalization of a well-known fact about the exact H\"older
regularity of stationary schemes in the stable case.

\smallskip \noindent In the proof of Theorem \ref{th:stable_case} we make use
of several auxiliary facts on long matrix products. The first one
of them is stated in the following lemma which is a special case
of~\cite[Proposition 2]{P}.

\begin{lemma}\label{l110}
Let $\mathcal{M}$ be a compact set of $d\times d$ matrices
 and $y \in \RR^d$. If $\rho(\mathcal M) > 1$ and $y$ does not
belong to a common invariant subspace of the matrices in
$\mathcal{M}$, then the sequence $\left\{ \displaystyle \max_{P_n
\in {\mathcal M}^n} \|P_n y \|, \ n \ge 1 \right\}$ diverges as $n \to \infty$.
\end{lemma}

\n Lemma \ref{l110} and the definition of the sequence $\left\{
\displaystyle \max_{P_n \in {\mathcal M}^n} \|P_n y \|, \ n \ge 1
\right\}$ yield

\begin{lemma}\label{l120}
Let $\mathcal{M}$ be a compact set of $d\times d$ matrices  and $y
\in \RR^d$. If $\rho(\mathcal M) > 1$ and $y$ does not belong to a
common invariant subspace of the matrices in~$\mathcal M$, then
for any $L \in \NN$ there exists $n \ge L$ such that
$$
\|M_1\cdots M_n y\| \, > \| y\|\ \quad \hbox{and} \quad \, \|M_1\cdots M_n y\| \,
> \|M_{n-i}\cdots M_n y\|\, , \ i = 0, \ldots , n-2,
$$
for $M_j \in \cM$.
\end{lemma}
\begin{proof} Let $L \in \NN$ and $C_L = \max \, \bigl\{\|P_j y\| \ \bigl|
\ P_j \in {\mathcal M}^j,  j \le L\, \bigr\}$. Then the shortest
product  $P_n \in {\mathcal M}^n$ such that $\|P_n y\|
> C_L$ (the set of such products is nonempty by Lemma~\ref{l110})
possesses the desired property and has its length bigger than~$m$.
\end{proof}

\n Next, we adapt Lemma \ref{l120} to the non-stationary setting.
The proof of the following result is similar to the proof of
Lemma~\ref{l120} and we omit it.

\begin{lemma}\label{l130} Let $\mathcal{M}$ and $\mathcal M^{\, (k)}$, $k \ge 1$,
be compact sets of $d\times d$ matrices  and $y \in \RR^d$. Assume
that $\rho(\mathcal M) > 1$,  the sequence $\{\mathcal M^{\, (k)},
\ k \ge 1\}$ converges to $\mathcal M$ and $y$ does not belong to
a common invariant subspace of the matrices in~$\mathcal M$. Then
there exists $L \in \NN$ and $C>0$ such that  for any $\tilde{L}
\ge L$ there exists $n \ge \tilde{L}$ such that, for $M_j \in
\cM^{(j+L-1)}$,
$$
 \, \|M_1\cdots M_n y\| \, > C\| y\|\ \quad \hbox{and}
 \quad \, \|M_1\cdots M_n y\| \,
> C \|M_{n-i}\cdots M_n y\|\, , \ i = 0, \ldots , n-2.
$$
\end{lemma}

\noindent We are  ready to  prove  Theorem
\ref{th:stable_case}.

\medskip
\begin{proof}[Proof of Theorem \ref{th:stable_case}]
Due to Corollary \ref{c10}, we only need to show that $\alpha \le
-\log_m \rho_{\mathbf a}$. Furthermore, by Lemma
\ref{lem:alpha_k=alpha_1}, it suffices to show that
$\alpha=\alpha_{\phi_n} \le -\log_m \rho_{\mathbf a}$ for some $n
\ge 1$. We choose an appropriate $n$ in the following way.
Firstly, $n$ should be such that
$$
 \rho(\{Q_\varepsilon^{(k)}, \ \varepsilon \in E, \ k \ge
 n\})< \rho_{\mathbf a}.
$$
(See \S \ref{subsec:transformation_basis} for the definition of
the matrices $Q_\varepsilon^{(k)}$.) Secondly, since by
assumption, there exists $\beta>0$ such that
$$
\displaystyle{\limsup_{k \rightarrow \infty} \delta_k^{1/k}} < \beta < \rho_\mathbf{a},
$$
thus, we can choose $n$ such that for any constant $C_0>0$ we have
$\delta_k < C_0 \beta^k$ for $k \ge n$. At the end of the proof we specify the particular constant $C_0$ needed for
our argument. Next, define
$$
 v(x)=\left( D^\nu \phi_n(x+\alpha) \right)_{\alpha \in K},
 \quad x \in [0,1]^s, \quad \nu \in \NN_0^s, \quad |\nu|=\ell.
$$
Let $k \ge 1$. By definition of $\phi_n$, for $x=\sum_{j=1}^k
\varepsilon_j m^{-j}$, $\varepsilon_j \in E$,  and $\|h\|_\infty
\le m^{-1}$, we have
\begin{equation} \label{aux_v}
 \Delta_{m^{-k}h} v(x)=  m^{\ell} T^{(n)}_{\varepsilon_1} \Delta_{m^{-k+1}h}
 v\left(\sum_{j=2}^k \varepsilon_j
m^{-j}\right)=m^{\ell k} T^{(n)}_{\varepsilon_1} \cdots
 T^{(n+k-1)}_{\varepsilon_k} \Delta_{h} v(0).
\end{equation}
By the same argument as in Proposition \ref{prop:convergence_Cl},
the first $\displaystyle L=\sum_{j=1}^{\ell+1} d_j$ components of
the vector $y:=\Delta_{h} v(0)$ are zero. Denote by
$\tilde{y}:=(y_{L+1}, \ldots, y_{|K|})^T$ the non-zero components
of $y$. W.l.o.g. we can assume that the vector $\tilde{y}$ does
not belong to any common invariant subspace of the matrices in
$\{T_{\varepsilon,\mathbf{a}}|_{V_\ell} \ : \ \varepsilon \in
E\}$. Otherwise, due to the stability of $\phi$ we have
$$
 \alpha_\phi=-\log_m \rho_\mathbf{a}=-\log_m \rho\{ T_{\varepsilon,\mathbf{a}}|_{W} \ : \ \varepsilon \in E\},
$$
where $W$ is the smallest subspace of $V_\ell$ such that it is
invariant under all operators in $\{ T_{\varepsilon,\mathbf{a}} \
: \ \varepsilon \in E\}$ and such that
$T_{\varepsilon,\mathbf{a}}|_{W}$, $\varepsilon \in E$, do not
have any common invariant subspace. For simplicity, we assume that
$W=V_\ell$, but the same argument we give below would apply, if
$W$ is a proper subspace of $V_\ell$. Let $r \in (\beta, \rho_\mathbf{a})$ be a real number.
The  sets
$$
 \cM=\{r^{-1}T_{\varepsilon,\mathbf{a}}|_{V_\ell}, \ \varepsilon \in E\} \quad
 \hbox{and} \quad \cM^{(k)}=\{r^{-1}Q^{(n+k-1)}_{\varepsilon}, \ \varepsilon \in E\},
 \quad k\ge 1,
$$
and the vector $\tilde{y}$  satisfy the assumptions of Lemma \ref{l130}. Thus,
we can appropriately modify $n$ chosen above  to get
\begin{eqnarray} \label{aux_r}
  &&\quad \|Q^{(n)}_{\varepsilon_1} \cdots Q^{(n+k-1)}_{\varepsilon_k} \tilde{y}\| \, > C r^k \| \tilde{y}\|\ \quad \hbox{and} \notag \\
 &&\quad \|Q^{(n)}_{\varepsilon_1} \cdots Q^{(n+k-1)}_{\varepsilon_k}  \tilde{y}\| \, > C r^{k-i}\|Q^{(n+k-i)}_{\varepsilon_{k-i+1}} \cdots Q^{(n+k-1)}_{\varepsilon_k}  \tilde{y}\|\, , \ i = 1, \ldots , k-1.
\end{eqnarray}

Denote by $H^{(n+k-i)}_j\in \RR^{ 1 \times |K|}$  the $j-$th row
of the matrix $T^{(n+k-i)}_{\varepsilon}$, $\varepsilon \in E$.
Define $y_0:=y$, the we have
\begin{equation}\label{Del1}
  T^{(n+k-1)}_{\varepsilon_k} y_0 \ = \left(\begin{array}{c} 0\\ \vdots \\ 0 \\ Q^{(n+k-1)}_{\varepsilon_{k}} \tilde{y}\end{array} \right) + \ \sum_{j=1}^{L} \langle H^{(n+k-1)}_j, y_0\rangle  e_j\, ,
\end{equation}
where $e_j$, $j=1, \ldots, L$, are the standard first $L$ unit
vectors of $\RR^{|K|}$ and $ \langle H^{(n+k-1)}_j, y_0\rangle$ is the scalar
product of the vectors $H^{(n+k-1)}_j$ and $y_0$. Define
$y_1:=\left(\begin{array}{cccc} 0 & \dots & 0 &
Q^{(n+k-1)}_{\varepsilon_{k}} \tilde{y} \end{array} \right)^T$.
Then, applying $T^{(n)}_{\varepsilon_1} \cdots
T^{(n+k-2)}_{\varepsilon_{k-1}}$ to both sides of \eqref{Del1}, we
get
$$
  T^{(n)}_{\varepsilon_1} \cdots T^{(n+k-2)}_{\varepsilon_{k-1}} y_1= T^{(n)}_{\varepsilon_1} \cdots T^{(n+k-1)}_{\varepsilon_k} y_0 \
  - \ \sum_{j=1}^{L} \langle H^{(n+k-1)}_j, y_0\rangle\, T^{(n)}_{\varepsilon_1} \cdots T^{(n+k-2)}_{\varepsilon_{k-1}} e_j\, ,
$$
and, thus, by triangle inequality,
$$
  \| T^{(n)}_{\varepsilon_1} \cdots T^{(n+k-1)}_{\varepsilon_k} y_0 \| \ge \| T^{(n)}_{\varepsilon_1} \cdots T^{(n+k-2)}_{\varepsilon_{k-1}} y_1\|
  - \ \sum_{j=1}^{L}| \langle H^{(n+k-1)}_j, y_0\rangle|\, \|T^{(n)}_{\varepsilon_1} \cdots T^{(n+k-2)}_{\varepsilon_{k-1}} e_j\|.
$$
Note that $n$ is such that, for any $n+k-i \ge n$, the matrix
$T^{(n+k-i)}_{\varepsilon}$ is bounded by the matrix~$R_{n+k-i}$, in the sense of
Lemma \ref{lemma:estimate_products_of_T}.
Then, due to the structure of $y_0$, we have $| \langle H^{(n+k-1)}_j, y_0\rangle| = {\cal O}( m^{-(\ell-j+1)(n+k-1)}\sigma_{n+k-1}) \|y_0
\|$. By Lemma \ref{lemma:estimate_products_of_T}, we also obtain
the estimate $\|T^{(n)}_{\varepsilon_1} \cdots
T^{(n+k-2)}_{\varepsilon_{k-1}} e_j \| = {\cal O}(m^{-(j-1)k})$,
$j=1, \ldots, L$. And, thus,
\begin{eqnarray*}
   \sum_{j=1}^{L}| \langle H^{(n+k-1)}_j, y_0\rangle|\, \|T^{(n)}_{\varepsilon_1} \cdots
   T^{(n+k-2)}_{\varepsilon_{k-1}} e_j\| & = &
   \sum_{j=1}^{L} {\cal O}( m^{-(\ell-j+1)(n+k-1)} \sigma_{n+k-1} m^{-(j-1)k}) \|y_0\|
   \\ &= &{\cal O}( m^{-\ell (n+k-1)} \sigma_{n+k-1}) \|y_0\|.
\end{eqnarray*}
The definition of $\sigma_k$ and the choice of $\beta$ yield ${\cal O}( m^{-\ell (n+k-1)} \sigma_{n+k-1})={\cal O}( \delta_{n+k-1}) <
\tilde{C} C_0 \beta^{n+k-1}$, $\tilde{C}>0$. Therefore,
$$
    \| T^{(n)}_{\varepsilon_1} \cdots T^{(n+k-1)}_{\varepsilon_k} y_0 \| \ge \| T^{(n)}_{\varepsilon_1} \cdots T^{(n+k-2)}_{\varepsilon_{k-1}} y_1\|
 - \tilde{C} C_0 \beta^{n+k-1} \|y_0\|.
$$
Set $y_i:=\left(\begin{array}{cccc} 0 & \dots & 0 &  Q^{(n+k-i)}_{\varepsilon_{k-i+1}} \cdots Q^{(n+k-1)}_{\varepsilon_{k}} \tilde{y} \end{array} \right)^T$, $i=2, \ldots, k$. Then, analogous successive argument for $\| T^{(n)}_{\varepsilon_1} \cdots T^{(n+k-i)}_{\varepsilon_{k-i+1}} y_{i-1}\|$,
$i=2, \ldots,k$, yields
$$
  \| T^{(n)}_{\varepsilon_1} \cdots T^{(n+k-1)}_{\varepsilon_k} y_0 \| \ge \|y_k\|- \tilde{C} C_0 \sum_{i=0}^{k-1} \beta^{n+k-i-1} \|y_i\|.
$$
From \eqref{aux_r} we get $\|y_i\|<r^{-k+i}C^{-1}\|y_k\|$, $i=0,
\ldots, k-1$, which implies
$$
  \| T^{(n)}_{\varepsilon_1} \cdots T^{(n+k-1)}_{\varepsilon_k} y_0 \| > \left(1-\frac{\tilde{C} C_0 \beta^{n-1}}{C}
  \sum_{i=0}^{k-1} \left(\frac{\beta}{r}\right)^{k-i} \right) \|y_k\| > \left(1-\frac{\tilde{C} C_0 \beta^{n-1}}{C (1-\frac{\beta}{r})} \right) \|y_k\|.
$$
In the second estimate above we used the fact that $\beta<r$.  Choose
$0<C_0<\frac{C(1-\frac{\beta}{r})}{\tilde{C} \beta^{n-1}}$ and define $C_1:=1-\frac{\tilde{C} C_0 \beta^{n-1}}{C (1-\frac{\beta}{r})}>0$.
Therefore, by \eqref{aux_r}, we have $\|y_k\|>C r^k \|y_0\|$ and, thus,
$$
  \| T^{(n)}_{\varepsilon_1} \cdots T^{(n+k-1)}_{\varepsilon_k} y_0 \| > C_1 \|y_k\| > C_1 C \, r^k \|y_0\|, \quad k \ge 1.
$$
Finally, this estimate and  \eqref{aux_v} yield $
 \|\Delta_{m^{-k}h}v(x)\| > C_1 C r^k m^{\ell k} \|y_0\|$, $k \ge 1$. Therefore, the H\"older exponents of all $D^\nu \phi_n$, $\nu \in \NN_0^s$,
$|\nu|=\ell$, are bounded from above by $-\ell-\log_m r$ and,
thus, $\alpha=\alpha_{\phi_n} \le -\log_m r$. Taking the limit as
$r$ goes to $\rho_{\mathbf a}$, we obtain the desired estimate
$\alpha \le -\log_m \rho_{\mathbf a}$.
\end{proof}

\noindent If the symbols of the scheme $\{ S_{{\mathbf a}^{(k)}}, \ k \ge 1\}$ satisfy
sum rules of order $\ell+1$, then we get the following immediate consequence of Theorem \ref{th:stable_case}.

\begin{corollary} \label{th:cor_stable_case}  Let $\ell \ge 0$. Assume  the
stationary scheme $S_{\mathbf{a}}$ is $C^\ell-$convergent with the
stable refinable basic limit function $\phi$ whose H\"older
exponent $ \alpha_\phi$ is $\ell \le \alpha_\phi<\ell+1$. If the
symbols of the scheme $\{ S_{{\mathbf a}^{(k)}}, \ k \ge 1\}$
satisfy sum rules of order $\ell+1$ and $\displaystyle \lim_{k
\rightarrow \infty} \mathbf{a}^{(k)}=\mathbf{a}$, then  $\{
S_{{\mathbf a}^{(k)}}, \ k \ge 1\}$ is $C^\ell-$convergent and the
H\"older exponent of its limit functions is also $\alpha_\phi$.
\end{corollary}

\subsection{Applications and examples} \label{subsec:convergence_examples}
\medskip\noindent In the this section, see \S
\ref{subsubsec:Daub_wavelets}, we prove the conjecture formulated
in \cite{DKLR}, which stipulates the H\"older regularity of the
generalized Daubechies wavelets. The proof of this conjecture is a
direct consequence of Theorem \ref{th:stable_case}. We also
determine the exact H\"older regularity of some of such
generalized Daubechies wavelets. Moreover, in \S
\ref{subsubsec:examples}, we illustrate our theoretical
convergence and H\"older regularity results with several
deliberately simple examples for which though neither the results
of \cite{ContiDynManniMazure13} nor the ones in \cite{DynLevin,
DynLevinYoon} are applicable.

Note that, in this section, we use the techniques from \cite{GugPro} that allow
for exact computation of the joint spectral radius of the corresponding matrix
sets. The method in \cite{GugPro} determines the so-called \emph{spectrum maximizing product}
of such sets, which yields the exact value of the joint spectral radius.

\begin{definition}
 Let $\cM$ be a compact collection of square matrices. The product $P:=M_1 \cdots M_m$, $M_j \in \cM$,
 is spectrum maximizing, if
 $\rho(\cM)=\rho(P)^{1/m}$, where $\rho(P)$ is the spectral radius of $P$.
\end{definition}

\subsubsection{Exact H\"older regularity of generalized Daubechies wavelets}
\label{subsubsec:Daub_wavelets}

\noindent The non-stationary Daubechies wavelets are defined and
studied in \cite{DKLR} and are obtained from Daubechies wavelets
in \cite{D} by suitable perturbation of the roots of the
stationary symbols. Let $n \ge 2$. To an arbitrary set $\Lambda_n:
= \{\lambda_0, \ldots , \lambda_{n-1}\}$ of real numbers
$\lambda_j$, $j=0, \ldots, n-1$, the authors in \cite{DKLR}
associate the generalized Daubechies wavelet
function~$\psi^{\Lambda_n}$. The corresponding refinable function
$$
 \phi^{\Lambda_n}:=\lim_{k \rightarrow \infty} S_{\mathbf{a}^{(k)}}
 S_{\mathbf{a}^{(k-1)}} \cdots S_{\mathbf{a}^{(1)}} \bdelta
$$
is the limit function of a non-stationary subdivision scheme
$\{S_{\mathbf{a}^{(k)}}, \ k \ge 1\}$ reproducing exponential
polynomials, i.e., solutions of the ODE of order $n$ with constant
coefficients and with spectrum $\Lambda_n$. The interested reader
can find more details on the construction and properties of these
wavelets~$\psi^{\Lambda_n}$, $n \ge 2$, in \cite{DKLR}.

\noindent Next we would like to mention the following two
properties of these masks $\{\mathbf{a}^{(k)}$, $k \ge 1\}$:

\begin{description}
\item[$(i)$] the sequence of masks $\{ \mathbf{a}^{(k)}, \ k \ge 1\}$
converges to the mask $\mathbf{m}_n$ of the classical $n-$th
Daubechies refinable function
$
 \varphi_n:=\lim_{k \rightarrow \infty} S_{\mathbf{m}_n}^k
\bdelta \, ;
$

\item[$(ii)$]  the corresponding  symbols $\{ a_*^{(k)}(z), \ k \ge 1\}$ satisfy
approximate sum rules of order $n$ with $\delta_k = {\cal
O}(2^{-nk}), \, k \ge 1$.
\end{description}

\smallskip \n In~\cite{DKLR} the authors estimated the H\"older exponent
of the generalized Daubechies wavelets and conjectured that it equals
to the H\"older exponent of the usual (stationary) Daubechies wavelets
(Conjecture~\ref{conj} stated in \S \ref{subsec:summary}). The following result proves this
conjecture.

\begin{theorem}\label{th50} Let $n\ge 2$. For every set
$\Lambda_n = \{\lambda_0, \ldots , \lambda_{n-1}\}$, the H\"older
regularity of the generalized Daubechies type wavelet
$\psi^{\Lambda_n}$ is equal to the H\"older regularity of the
classical Daubechies wavelet~$\psi_n$ derived from $\displaystyle
\varphi_n$.
\end{theorem}

\begin{proof} We invoke  Theorem \ref{th:stable_case}.
Since a  compactly supported wavelet function has the same regularity as
the corresponding refinable function, we need to show that the
functions $\phi^{\Lambda_n}$ and $\varphi_{n}$ have the same
regularity. The non-stationary subdivision scheme
$\{S_{\mathbf{a}^{(k)}}, \ k \ge 1\}$ generating
$\phi^{\Lambda_n}$ satisfies the assumptions of Theorem
\ref{th:main_Hoelder} with $\ell=n-1$ and $\cA=\{\mathbf{m}_n\}$.
Indeed, the masks of the scheme  $\{S_{\mathbf{a}^{(k)}}, \ k \ge
1\}$ are constructed in \cite{DKLR} in such a way that they
converge to the mask $\mathbf{m}_n$. The Daubechies refinable
function $\varphi_n$ is stable and, hence, its H\"older exponent is $\alpha_{\varphi_n} = - \log_2 \rho_\cA$. It is well-known that
$\alpha_{\varphi_n} < n$, therefore $\rho_\cA > 2^{-n}$. Thus, by
$(ii)$ we have  $\displaystyle \limsup_{k \to \infty}
\delta_k^{1/k} \le 2^{-n} < \rho_\cA$. Therefore, all assumptions
of Theorem \ref{th:stable_case} are satisfied and the H\"older
exponent $\alpha$ of $\phi^{\Lambda_n}$ satisfies $\alpha =
\alpha_{\varphi_n}=- \log_2 \rho_{\cA}$.
\end{proof}

In~\cite{DKLR} the H\"older exponent $\alpha$  is estimated by the
rate of decay of the Fourier transform
$\hat{\phi}^{\Lambda_n}$ of $\phi^{\Lambda_n}$. It is well-known
that for any continuous, compactly supported function $f$, its
H\"older exponent $\alpha_f$ satisfies
$$
 \eta(f) -1 \le \alpha_f \le \eta(f), \quad \eta(f) = \sup\, \bigl\{\beta \ge 0 \
 : \ |{\widehat f}(\omega)|
\le C(1 + |\omega|)^{-\beta}  \, , \  \omega \in \RR\, \bigr\},
$$
and this gap of length~$1$ is, in general, unavoidable~\cite{Z}.
In \cite[Theorem 29]{DKLR} the authors show that
$\eta(\phi^{\Lambda_n}) \ge \eta(\varphi_{n})$, which, thus,
implies the following lower bound for the H\"older exponent
$\alpha$ of $\phi^{\Lambda_n}$
$
 \alpha \, \ge \, \eta(\varphi_{n}) -1.
$
Using lower bounds for the values $\eta(\varphi_{n})$ known from
the literature, one can estimate the regularity of the generalized
Daubechies wavelets. Table in \eqref{tab1} compares those rough
bounds given in~\cite{D} (computed  by the method of invariant
cycles) with the exact values of $\alpha= -\log_2 \rho_\cA$, which
we compute using the techniques in \cite{GugPro}.

\begin{equation}
\begin{array}{|c|c|c}
n&  \eta(\varphi_{n}) - 1 & \alpha=-\log_2 \rho_\cA \\
\hline
2   & 0.339 & 0.5500 \\
3   & 0.636 & 1.0878 \\
4   & 0.913 & 1.6179 \\
5   & 1.177 & 1.9690 \\
6   & 1.432 & 2.1891 \\
7   & 1.682 & 2.4604 \\
8   & 1.927 & 2.7608 \\
9   & 2.168 & 3.0736 \\
10  & 2.406 & 3.3614 \\
\end{array}
\label{tab1}
\end{equation}

\subsubsection{Further examples} \label{subsubsec:examples}

In this subsection we apply our convergent and regularity results
to several deliberately simple non-stationary subdivision schemes
whose analysis was impossible so far. These examples are constructed only for
illustration purposes.

\smallskip {\bf Example 2:} We start with a non-stationary subdivision scheme with
a general dilation matrix $M$ and masks which are level dependent
convex combination of two multivariate masks
$\mathbf{a},\mathbf{b} \in \ell_0(\ZZ^s)$. We assume that
$\mathbf{a}$ defines a (stationary) convergent subdivision scheme
and that $\mathbf{b}$ satisfies sum rules of order $1$. Convex
combinations of such subdivision masks were also investigated in
\cite{CharinaContiJetterZimm11,Conti08}. In particular, we define the
non-stationary subdivision scheme $\{S_{\mathbf{a}^{(k)}}, \ k \ge
1 \}$  by
\begin{equation}\label{es:1}
 \mathbf{a}^{(k)}:=\left(1-\frac{1}{k}\right)\mathbf{a}+\frac{1}{k} \mathbf{b}, \quad k \ge 1.
\end{equation}
This non-stationary scheme does not satisfy
the condition in \eqref{def:AsEq} for $\ell=0$, since $|\ra^{(k)}(\alpha)-\ra(\alpha)| =
  |\rb(\alpha)-\ra(\alpha)| \frac{1}{k}$,
$$
 \sum_{k \in \NN}\max_{\varepsilon \in E}
  \left\{ \sum_{\alpha \in \ZZ^s} |\ra^{(k)}(\varepsilon+M\alpha)-\ra(\varepsilon+M\alpha)| \right\}\nless \infty
$$
Nevertheless, $\{S_{\mathbf{a}^{(k)}}, \ k \ge 1 \}$ satisfies the assumptions
of Theorem \ref{teo:convergence}, since, by construction, all symbols satisfy
approximate sum rules of order $1$ and $\displaystyle \lim_{k \rightarrow
\infty} \mathbf{a}^{(k)}=\mathbf{a}$. Therefore, we are able to conclude that
the scheme is at least $C^0-$convergent. Moreover, in the case $M=mI$, the
assumptions that $S_\ab$ is $C^\ell$-convergent and that $\mathbf{b}$ satisfies
sum rules of order $\ell+1$, imply, by Theorem \ref{th:main_Hoelder}, that the
H\"older regularity of the scheme in (\ref{es:1}) is at least as high as for
$S_\ab$. Indeed, for $s=2$ and $M=2I$, let $\mathbf{a}$ be the mask of the
butterfly scheme (\cite{GregoryDynLevin} with $\omega=1/16$) and $\mathbf{b}$ be the mask of
the Courant element, the box spline  $B_{111}$. Then, using the method in \cite{GugPro}, we compute
$\rho(\cT_{\mathbf{a}}|_{V_1})=1/4$ and, thus, the scheme $\{S_{\mathbf{a}^{(k)}}, \ k \ge 1 \}$
is $C^1-$convergent and its H\"older exponent is $\alpha
= 2$.

\smallskip \noindent In the next example we construct non-stationary schemes
with sets of limit points $\cA$ of cardinality $2$.

\smallskip  {\bf Example 3:}  Let
$\cal I\subset \NN$ be some infinite set, such that $\NN \setminus {\cal I}$ is also infinite. We
consider the non-stationary scheme with the masks
\begin{equation}\label{es:2}
 \ab^{(k)}:=\left\{
             \begin{array}{ll}
               \mathbf{a},\ & k\in {\cal I}, \\
                \cb,\ & k\in  \NN\setminus {\cal I},
             \end{array}
           \right. \quad k \ge 1.
\end{equation}
We assume that the masks $\mathbf{a},\ \cb\in \ell_0(\ZZ^s)$
define stationary convergent subdivision schemes with the same dilation matrix
$M$. Moreover, we assume that $\rho\left(\cT_\cA|_{V_0}\right)<1$,
$\cA=\{\mathbf{a}, \mathbf{c}\}$. Here the
notion of asymptotic equivalence is not applicable, but Theorem
\ref{teo:convergence}   allows us to
establish $C^0-$convergence of the scheme in \eqref{es:2}. If $M=mI$ and
$\mathbf{a},\ \cb$ are such that
$\rho\left(\cT_\cA|_{V_\ell}\right)<m^{-\ell}$,
Theorem \ref{th:main_Hoelder} also yields a lower bound for the H\"older regularity of
$\{S_{\mathbf{a}^{(k)}}, \ k \ge 1\}$.

\smallskip \noindent  For example, for $s=1$ and $M=2$,  let
$$
a_*(z):=\frac18(1+z)^4 \quad \mbox{and} \quad
c_*(z):=\frac1{16}(-1+9z^2+16z^3+9z^4-z^6), \quad z \in \CC \setminus 0,
$$
be the symbols of the cubic B-spline and the $4$-point scheme
(\cite{DesDubuc} with $\omega=1/16$), respectively. Using the method in
\cite{GugPro} we obtain
$\rho(\cT_{\cA}|_{V_1})=0.35385...$. This tells us that the
corresponding scheme $\{S_{\mathbf{a}^{(k)}}, \ k \ge 1\}$ has the
H\"older exponent $\alpha \ge 1.49876...$. For the
computation of $\rho(\cT_{\cA}|_{V_1})$, we used the set
$\cT_{\cA}|_{V_1}=\{T_1:=T_{0,\mathbf{c}}|_{V_1}, T_2 :=T_{1,\mathbf{c}}|_{V_1},T_3 :=T_{0,\mathbf{a}}|_{V_1},T_4:=T_{1,\mathbf{a}}|_{V_1} \}$ with
\begin{eqnarray*}
&& \hskip -0.7cm T_1= \left(
\begin{array}{rrrr}
      \frac{1}{8}  &  \frac{1}{8} &        0 &        0 \\
     -\frac{1}{16}  &  \frac{3}{8} &  -\frac{1}{16} &        0 \\
           0  &  \frac{1}{8} &   \frac{1}{8} &        0 \\
           0  & -\frac{1}{16} &   \frac{3}{8} &  -\frac{1}{16}
             \end{array} \right), \quad
   T_2= \left( \begin{array}{rrrr}
     -\frac{1}{16}  &  \frac{3}{8} &  -\frac{1}{16} &        0 \\
           0  &  \frac{1}{8} &   \frac{1}{8} &        0 \\
           0  & -\frac{1}{16} &   \frac{3}{8} &  -\frac{1}{16} \\
           0  &       0 &   \frac{1}{8} &   \frac{1}{8}
             \end{array} \right),
\\
&& \hskip -0.7cm T_3= \left(
\begin{array}{rrrr}
      \frac{1}{8}  &       0 &        0 &        0 \\
      \frac{1}{8}  &  \frac{1}{4} &   \frac{1}{8} &        0 \\
           0  &       0 &   \frac{1}{8} &   \frac{1}{4} \\
           0  &       0 &        0 &        0
             \end{array} \right), \quad
   T_4= \left( \begin{array}{rrrr}
      \frac{1}{4} &   \frac{1}{8} &        0 &        0 \\
           0 &   \frac{1}{8} &   \frac{1}{4} &   \frac{1}{8} \\
           0 &        0 &        0 &   \frac{1}{8} \\
           0 &        0 &        0 &        0
             \end{array} \right).
\end{eqnarray*}
\normalsize The spectrum maximizing product we obtain is $T_1 (T_1
T_3)^{13}$. If, instead of the mask $\mathbf{a}$ above, we take  the mask of the quadratic B-spline, then we
obtain $\rho(\cT_{\cA}|_{V_1})=0.35045...$, which tells us
that the corresponding scheme $\{S_{\mathbf{a}^{(k)}}, \ k \ge
1\}$ has H\"older exponent $\alpha \ge 1.51271...$. For the
computation of $\rho(\cT_{\cA}|_{V_1})$ we used the set
$\cT_{\cA}|_{V_1}=\{T_1, T_2, T_5,T_6\}$ with
\begin{eqnarray*}
&& T_5: = T_{0,\mathbf{a}}|_{V_1}=\left( \begin{array}{rrrr}
      \frac{1}{4}  &       0 &        0 &        0 \\
           0  &  \frac{1}{4} &   \frac{1}{4} &        0 \\
           0  &       0 &        0 &   \frac{1}{4} \\
           0  &       0 &        0 &        0 \\
             \end{array} \right), \qquad
   T_6 := T_{1,\mathbf{a}}|_{V_1}=\left( \begin{array}{rrrr}
      \frac{1}{4}  &  \frac{1}{4} &        0 &        0 \\
           0  &       0 &   \frac{1}{4} &   \frac{1}{4} \\
           0  &       0 &        0 &        0 \\
           0  &       0 &        0 &        0
             \end{array} \right).
\end{eqnarray*}
\normalsize The spectrum maximizing product is $T_1 (T_1 T_5)^{2}$.
Note that $\rho(\cT_{\cA}|_{V_1})$ can be bigger than either $\rho(\cT_{\ab}|_{V_1})$
or $\rho(\cT_{\mathbf{c}}|_{V_1})$. It is also of interest that  $\rho(\cT_{\cA}|_{V_1})$
decreases, if we replace the mask of the cubic B-spline by the mask of the less regular scheme
corresponding to the quadratic B-spline.

\smallskip \noindent The next example studies a univariate ternary ($M=3$)
non-stationary scheme.

\smallskip  {\bf Example 4:} We consider the
alternating sequence of symbols
\begin{equation}\label{es:3}
 a_*^{(k)}(z):=\left\{
             \begin{array}{ll}
               c_*^{(k)}(z):=z^{-6} K_1^{(k)} \  (z^2 + z + 1)^2 (z + 1)
{\tilde c}_*^{(k)}(z), & k\ \ \hbox{even,} \\ \\
              d_*^{(k)}(z):= z^{-6} K_2^{(k)} \  (z^2 + z + 1)^2 (z + 1)
{\tilde d}_*^{(k)}(z), & k\ \ \hbox{odd,}
             \end{array} \right. \quad k\ge 1\,,
\end{equation}
where  $K_1^{(k)}$ and $K_2^{(k)}$ are suitable normalization
constants and the factors ${\tilde c}_*^{(k)}(z)$ and ${\tilde
d}_*^{(k)}(z)$ are defined by
$$\begin{array}{ll}{\tilde c}_*^{(k)}(z):=&
 \Big( z^4 + (4(w^{(k)})^2 - 2)z^3 + (16(w^{(k)})^4 - 16(w^{(k)})^2 + 3)z^2 +
(4(w^{(k)})^2 - 2)z + 1 \Big) \cdot \\
&\Big( (16(w^{(k)})^4 + 16(w^{(k)})^3 + 3)z^2 + (- 64(w^{(k)})^6 - 64(w^{(k)})^5 +
32(w^{(k)})^4 + \\
&\quad 32(w^{(k)})^3 - 12(w^{(k)})^2 - 12w^{(k)} - 6)z + 16(w^{(k)})^4 +
16(w^{(k)})^3 + 3 \Big), \end{array}$$ and
$${\tilde d}_*^{(k)}(z):=
 \Big(z^2 + z + 1 \Big) \\
\Big( w^{(k)}z^4 + 2w^{(k)}z^3 + (4(w^{(k)})^2-1)z^2 +2w^{(k)}z + 1 \Big),
$$
with $w^{(k)}:=\frac{1}{2}\left( e^{3^{-(k+1)}\lambda
/2}+e^{-3^{-(k+1)}\lambda /2}\right)$ and $\lambda\in \RR^+\cup
i\RR^+$.

\smallskip \noindent
The corresponding  non-stationary subdivision scheme $\{S_{{\bc}^{(k)}},\ k\ge
1\}$ was considered in \cite{CharinaContiRomani13,ContiRomani13}.
In \cite{ContiRomani13}, the authors investigate the convergence
of the sequence of symbols $\{{c}_*^{(k)}(z),\ k\ge 1\}$ to the
symbol
$$
{c}_*(z)=-z^{-6}\frac{1}{1296}(z^2 + z + 1)^4(z + 1)(35z^2 - 94z
+35)
$$
of the ternary dual stationary 4-point Dubuc-Deslaurier scheme,
which is known to be at least $C^2$-convergent. The sequence of the symbols of
the non-stationary subdivision scheme $\{S_{{\bd}^{(k)}},\ k\ge
1\}$ converges to the symbol
$$
{d}_*(z)=-z^{-6}\frac{1}{162}(z^2 + z + 1)^5(z + 1),
$$
see
\cite{CharinaContiRomani13}. The stationary scheme $S_\bd$ is
known to be at least $C^2-$convergent. We would like to remark
that $\{S_{{ \bc}^{(k)}},\ k\ge 1\}$ and $\{S_{{ \bd}^{(k)}},\
k\ge 1\}$ are both schemes generating/reproducing certain spaces of
exponential polynomials, see \cite{ContiRomani11}.

Using Theorem \ref{th:main_Hoelder} with $\cA=\{\cb, \db\}$ and the method in
\cite{GugPro}, we determine a  lower bound for the H\"older regularity of the
scheme $\{S_{\mathbf{a}^{(k)}}, \ k \ge 1\}$. Since we get
$\rho({\cT}_{\cA}|_{V_2})=0.04958...$, the corresponding H\"older
exponent satisfies $\alpha\ge 2.73437...$. For the computation of
$\rho({\cT}_{\cA}|_{V_2})$ we use the set
$\cT_{\cA}=\{T_1,T_2,T_3,T_4,T_5,T_6\}$ with $T_j:=1296\,T_{j-1,\mathbf{c}}|_{V_2}$,
$j=1,2,3$,
\begin{eqnarray*}
 \scriptsize T_1=\left(
\begin{array}{rrrr}
      35 & 0 & 0 \\
      -83  & -83 &  -24  \\
       0  &   35 &  -24  \\
             \end{array} \right), \
   T_2=\left( \begin{array}{rrrr}
      -24 & 35 & 0 \\
      -24 & -83  & -83  \\
       0& 0  &   35    \\
             \end{array} \right), \
   T_3=\left( \begin{array}{rrrr}
      -83&-24 & 35    \\
      35& -24 & -83    \\
      0&  0& 0
             \end{array} \right),
\end{eqnarray*}
and with $T_j:=T_{j-4,\mathbf{d}}|_{V_2}$, $j=4,5,6$,
\begin{eqnarray*} \scriptsize
T_4=\frac{1}{162}\left( \begin{array}{rrrr}
      1 & 0 & 0  \\
      5  & 5 &  3  \\
      0  &  1 & 3
  \end{array} \right), \quad
   T_5=\frac{1}{162}\left( \begin{array}{rrrr}
      3 & 1 & 0   \\
      3 & 5  & 5  \\
      0& 0  &  0
             \end{array} \right), \quad
   T_6=\frac{1}{162}\left( \begin{array}{rrrr}
      5& 3 & 1  \\
      1& 3 & 5   \\
      0&  0& 0
             \end{array} \right).
\end{eqnarray*}

The spectrum maximizing product is $T_1 T_3$, which implies that the H\"older
exponent $\alpha$ coincides with the H\"older exponent of the scheme $S_\cb$.

\smallskip \n In the next two examples, we construct and analyze the regularity of
a univariate and a multivariate non-stationary
subdivision schemes obtained by
suitable perturbations of the
masks of the known stationary subdivision schemes.
These non-stationary schemes are not asymptotically equivalent to any stationary
scheme and, thus, the results of \cite{DynLevin} are not applicable. Note though
that these schemes satisfy approximate sum rules of order $2$ and the other
assumptions of Theorem \ref{th:main_Hoelder}.

\smallskip {\bf Example 5:} For $s=1$ and
$M=2$, we  consider the sequence of masks
$\{\mathbf{a}^{(k)},\ k\ge 1\}$ with
\begin{equation}\label{eq:chaikinter}
\mathbf{a}^{(k)}:=\left\{\left(\frac14 -\frac1{k}\right),\
\left(\frac34 -\frac1{k}+2^{-2k}\right),\ \left(\frac34
+\frac1{k}\right),\ \left(\frac14
+\frac1{k}+2^{-2k}\right)\right\},\quad k\ge 1\,.
\end{equation}
Obviously, $\displaystyle \lim_{k \rightarrow \infty}
\mathbf{a}^{(k)}=\mathbf{a}$, where $\mathbf{a}=\left\{\frac14,\
\frac34,\ \frac34,\, \frac14 \right\}$ is the mask of the Chaikin
subdivision scheme \cite{Chaikin}.  It is easy to check that the symbols of this
non-stationary scheme satisfy
$$
 a_*^{(k)}(1)-2=2^{-2k+1}, \quad a_*^{(k)}(-1)=-2^{-2k+1}\quad \hbox{and} \quad
 D a_*^{(k)}(-1)= 2^{-2k+2}, \quad  k\ge 1\,,
$$
i.e. $\mu_k=\delta_k=2^{-2k+1}$ and, thus, the symbols satisfy
approximate sum rules of order $2$. To be able to apply Theorem
\ref{th:main_Hoelder}, we need to rescale the masks
$\mathbf{a}^{(k)}$ so that $\mu_k=0$, $k\ge 1$. It is easily done
by multiplying each of the masks $\mathbf{a}^{(k)}$ by the factor
$2/(2+\mu_k)$. After this modification the sequence $\{\delta_k, \
k \ge 1\}$ is still summable, since
$\displaystyle{
 \sum_{k \in \NN} \frac{2 \delta_k}{2+\mu_k} < \sum_{k \in \NN}
 \delta_k< \infty}.
$
Hence, by Theorem \ref{th:main_Hoelder} and the known fact that
$\rho(\cT_\ab|_{V_1})=\frac14$, the non-stationary scheme with masks in
\eqref{eq:chaikinter} is $C^1-$convergent with $\alpha = 2$.

\smallskip {\bf Example 6:} For
$s=2$ and $M=2I$, we consider the sequence of masks
$\{\mathbf{a}^{(k)},\ k\ge 1\}$ with for $k\ge1$
\begin{equation}\label{es:4}
\mathbf{a}^{(k)}=\frac1{16}\left(
             \begin{array}{ccccc}
               0 & 2^{-2k} & 1+\frac1k & 2-\frac1k & 1-\frac1k \\
               -\frac1k & 2-\frac1k+2^{-2k} & 6+\frac1{k}+2^{-2k} & 6+\frac1k+2^{-2k} & 2+2^{-2k} \\
               1-\frac1{k} & 6+\frac1k & 10+\frac2k & 6+\frac1k & 1-\frac1{k} \\
              2+ 2^{-k} & 6+\frac1{k}+2^{-2k} & 6+\frac1k+2^{-2k} & 2-\frac1{k}+2^{-k} & -\frac1k \\
               1-\frac1k & 2-\frac1k & 1+\frac1k & 2^{-2k} & 0 \\
             \end{array}
           \right)
\end{equation}
Obviously, $\displaystyle \lim_{k \rightarrow \infty}
\mathbf{a}^{(k)}=\mathbf{a}$, where $\mathbf{a}$ is the mask of the Loop
subdivision scheme \cite{Loop}.  Note that the symbols of this non-stationary scheme satisfy
approximate sum rules of order $2$, since, we have $\mu_k=5\cdot 2^{-(2k+4)}$
and $\delta_k=6\cdot 2^{-(2k+4)}$, $k \ge 1$. It is well-known
that $\rho(\cT_\ab|_{V_1 })=\frac14$. Thus, after an appropriate normalization
of the masks, by Theorem \ref{th:main_Hoelder}, we get that the non-stationary scheme is
$C^1-$convergent with the H\"older exponent $\alpha= 2$.

\section{Further properties} \label{sec:necessary}

In this subsection, we prove Theorem \ref{th:neccesary_conditions}
stated in Subsection \ref{subsec:summary}. Its proof is based on the next
Proposition \ref{p30} that studies the infinite products of
certain trigonometric polynomials. The statement of Proposition
\ref{p30} involves the following concepts.

\begin{definition} \label{def:simmetric_roots}
A pair of complex numbers $\{z, -z\}$ is called {\em a pair of
symmetric roots} of the algebraic polynomial $q$, if $q(z) = q(-z)
= 0$.
\end{definition}

Let $\{ q_k, \ k \ge 1\}$ be a sequence of algebraic polynomials
of degree~$N$ and define the function
\begin{equation}\label{prod1}
f(x):= \ \prod_{k=1}^{\infty} p_k(2^{-k}x)\, , \quad
p_k(x):=q_k(e^{-2\pi i x}), \quad  x \in \mathbb{R}\, .
\end{equation}
By \cite{CollelaH}, if a sequence of trigonometric polynomials
$\{p_k, \ k \ge 1 \}$ is bounded, then this infinite product
converges uniformly on each compact subset of~$\mathbb{R}$, and
hence, $f$ is analytic. Possible rates of decay of such functions as $x \to \infty$ was studied in~\cite{P1}.

\begin{proposition}\label{p30}
Assume that the sequence of trigonometric polynomials $\{p_k, \ k
\ge 1\}$ with $p_k(0)=1$, $k \ge 1$,  converges to a trigonometric
polynomial~$p$ that has no symmetric roots on $\RR$. If the
function $f$ in \eqref{prod1} satisfies $f(x) = o(x^{-\ell})$ for
$\ell \ge 0$ and $x \to +\infty$, then
$\delta_{k} =
o(2^{-\ell k})$ as $k \to \infty$, where
$$
 \delta_k=\max_{j=0, \ldots, \ell} 2^{-jk} \frac{|D^j p_k(1/2)|}{j!}, \quad k \ge 1.
$$
\end{proposition}

\begin{proof} By assumption $f(x) = o(x^{-\ell})$ for points of the form~$x = 2^{k-1} d + t$,
where $d$ is a fixed natural number, $t$ is an arbitrary number from~$[0, \sigma]$, $\sigma >0$,
 and $k \to \infty$.
Next, we choose these parameters $d \in \mathbb{N}$ and $\sigma > 0$ in a special way.\\
Firstly, we define $\sigma$. Since $\{p_k, \ k \ge 1\}$ converges
to~$p$, the sequence~$\{p_k, \ k \ge 1\}$ is bounded. Moreover,
$p_k(0)=1$, $k \ge 1$, implies that $f(0)=1$. This implies that
there are $\sigma \in (0, 1)$ and $C_0 > 0$ such that  for every
$r \ge 0$ and $R \in \mathbb{N}\cup \{\infty\}$ we have
\begin{equation}\label{sigma}
\Bigl|\, \prod_{j=1}^{R} p_{j+r}\bigl(2^{-j}t\bigr)\, \Bigr| \ \ge
\ C_0\, , \quad t \in [0, \sigma]\, .
\end{equation}
Next we choose the number~$d$. To this end we consider the binary
tree defined as follows: the number $1/2$ is at the root, the
numbers $1/4$ and $3/4$ are its children, and so on. Every
vertex~$\alpha$ has two children $\alpha/2$ and $(\alpha +1)/2$.
For convenience we shall identify a vertex and the corresponding
number. Thus, all vertices of the tree are dyadic points from the
interval~$(0,1)$. Indeed, the $n-$th level of the tree (i.e., the
set of vertices with the distance to the root equal to~$n$)
consists
of points~$2^{-n-1}j$, where $j$ is an odd number  from~$1$ to $2^{n+1}-1$.\\
The trigonometric polynomial~$p$ is $1-$periodic and, thus,  has
at most $N$ zeros in $[0, 1)$, and hence, on the tree. Therefore,
there is a number~$q$ such that all roots of~$p$ on the tree are
contained on levels $j \le q$. Since the polynomial~$p$ has no
symmetric roots, at least one of the two children of any vertex of
the tree   is not a root of~$p$. Whence, there is a path of
length~$q$ along the tree starting at the root (all paths are
without backtracking) that does not contain any root of~$p$. Let
$2^{-q-1}d$ be the final vertex of that path, $d$ is an odd
number,
 $1\le d \le 2^{q+1}-1$. Denote as usual by~$\{x\}$ the fractional part of~$x$. Then
the sequence~$\{2^{-1}d\}, \ldots , \{2^{-q-1}d\}$ does not
contain roots of~$p$. The sequence $\{2^{-q-2}d\}, \{2^{-q-3}d\},
\ldots $ does not contain them either, because there are no roots
of~$p$ on levels bigger than~$q$. Let~$n$ be the smallest natural
number such that $2^{-q-n-1}d < \sigma /2$. We have
$p(2^{-1}d)\cdots p(2^{-q-n-1}d) \ne 0$. Since~$p_k \to p$ as $k \to
\infty$, and all $p_k$ are equi-continuous on~$\mathbb{R}$, it
follows that there is a constant $C_1>0$ such that
\begin{equation}\label{d}
\Bigl|\, \prod_{j=1}^{q+n} p_{k+j}\bigl(2^{-j-1}d +
2^{-k-j}x\bigr)\, \Bigr| \ \ge \ C_1\, , \quad x \in [0,
\sigma]\,,
\end{equation}
for  sufficiently large~$k$. Now we are ready to estimate the
value~$f(2^{k-1}d + t)$. We have
\begin{eqnarray*}
&&\Bigl|\, f\bigl(2^{k-1}d + t\bigr)\, \Bigr| \ = \ \Bigl|\,
\prod_{j=1}^{k-1} p_{j}\bigl(2^{k-1-j}d + 2^{-j}t\bigr)\, \Bigr|
\times \, \Bigl| p_k \bigl( 2^{-1}d + 2^{-k}t \bigr)\, \Bigr|
\times \\
&&\times \Bigl|\, \prod_{j=1}^{q+n} p_{k+j}\bigl(2^{-j-1}d +
2^{-k-j}t\bigr)\Bigr| \, \times \, \Bigl|\, \prod_{j=1}^{\infty}
p_{k+q+n+j}\bigl(2^{-j}(2^{-q-n-1}d + 2^{-k-q-n}t)\bigr)\Bigr|\, .
\end{eqnarray*}
To estimate the first factor in this product, we note that $2^{k-1-j}d \in
\mathbb{Z}$, whenever~$j \le k-1$, and hence
$p_{j}\bigl(2^{k-1-j}d + 2^{-j}t\bigr) = p_j(2^{-j}t)$. Thus, the
first factor is $\bigl|\, \prod_{j=1}^{k-1}
p_{j}(2^{-j}t)\bigr|$, which is,
by~(\ref{sigma}), bigger than or equal to~$C_0$, for every~$t \in [0, \sigma]$.\\
The third factor $\bigl|\, \prod_{j=1}^{q+n} p_{k+j}(2^{-j-1}d +
2^{-k-j}t)\bigr|$, by~(\ref{d}), is at least~$C_1$. Finally, the
last factor is bigger than or equal to~$C_0$. To see this it
suffices to use~(\ref{sigma}) for~$R = \infty, r = k+q+n, x =
2^{-q-n}d + 2^{-k-q-n}t$ and note that~$x < \sigma$ by the choice
of~$n$. Thus,
  $$
|f(2^{k-1}d + t)| \ \ge \ C_0^2C_1 \bigl| p_k \bigl( 2^{-1}d +
2^{-k}t \bigr)\bigr|\, .
$$
On the other hand, by assumption, $f(2^{k-1}d + t) = o(2^{-\ell
k})$ as $k \to \infty$, consequently
 $ p_k ( 2^{-1}d + 2^{-k}t) = o(2^{-\ell k})$. The number $d$ is odd, hence, by periodicity,~$p_k ( 2^{-1}d + 2^{-k}t) =
 p_k ( 1/2 + 2^{-k}t)$. Thus, we arrive at the following asymptotic relation: for every $t \in [0, \sigma]$ we have
 \begin{equation}\label{1/2}
 p_k \bigl( 1/2 + 2^{-k}t\bigr) \ = \ o\, \bigl( 2^{-\ell k}\bigr)\quad \mbox{as}\ k \to \infty\, .
 \end{equation}
 This already implies that $D^j p_k(1/2) = o(2^{(j-\ell)k})$ as $k \to \infty$, for every~$j = 0, \ldots , \ell$.
 Indeed, consider the Tailor expansion of the function $h(t) = p_k \bigl( 1/2 + 2^{-k}t\bigr)$
  at the point~$0$ with the remainder in Lagrange form:
 $$
 h(t) \ = \ \sum_{j=0}^{\ell}\frac{D^j h(0)}{j!}\, t^j\ + \ \frac{D^{\ell+1} h(\theta)}{(\ell+1)!}\, t^{\, \ell+1}\, , \ t \in [0, \sigma]\, ,
 $$
where $\theta = \theta(t) \in [0, t]$. Substituting~$D^j h(0) =
2^{-jk} D^j p_k(1/2)$, we get
$$
 p_k(1/2 + 2^{-k}t) \ = \ \sum_{j=0}^{\ell}\frac{D^j p_k(1/2)}{j!}\, 2^{-jk}\, t^j\ + \ \frac{D^{\ell+1}p_k(1/2 + 2^{-k}\theta)}{(\ell+1)!}\, 2^{-(\ell+1)k}\, t^{\, \ell+1}\, , \ t \in [0, \sigma]\, .
$$
First, we estimate the remainder. Since the sequence of
trigonometric polynomials $\{p_k, \ k \ge 1\}$ is
bounded, the norms $\|D^{\ell+1} p_k \|_{C[0, \sigma]}$ do not
exceed some constant~$C_2$. Therefore,
$$
\left|\frac{D^{\ell+1} p_k(1/2 + 2^{-k}\theta)}{(\ell+1)!}\,
2^{-(\ell+1)k}\, t^{\, \ell+1}\right| \ \le \
\frac{C_2}{(\ell+1)!}2^{-(\ell+1)k}\, \sigma^{\, \ell+1} \ = \
o(2^{-\ell k})\quad  \mbox{as} \ k \to \infty\, .
$$
 Combining this with~(\ref{1/2}), we get
\begin{equation}\label{polynom}
\left\|\, \sum_{j=0}^{\ell}\frac{D^j p_k(1/2)}{j!}\, 2^{-jk}\,
t^j\, \right\|_{\, C([0, \sigma])}\ = \ o(2^{-\ell k})\quad
\mbox{as} \ k \to \infty\, .
\end{equation}
Since, in a finite-dimensional space, all norms are equivalent,
the norm of an algebraic polynomial of degree~$\ell$ in the space
$C([0, \sigma])$ is equivalent to its largest coefficient.
Whence,~(\ref{polynom}) implies that
\begin{equation} \label{eq:aux}
\max\limits_{j=0, \ldots ,
\ell}  2^{-jk} \frac{|D^j p_k(1/2)|}{j!} = o(2^{-\ell k}), \quad k
\to \infty.
\end{equation}
\end{proof}

\smallskip
\noindent We are finally ready to prove the  main result of this
section, Theorem \ref{th:neccesary_conditions}.


\medskip \begin{proof}[Proof of Theorem \ref{th:neccesary_conditions}]  Let $p_k(\omega):= a_*^{(k)}(e^{-2\pi i \omega})$, $\omega \in \RR$, be the symbol of the $k-$th mask
in the trigonometric form. If the non-stationary scheme converges
to a continuous compactly supported refinable function~$\phi$,
then its Fourier transform $\widehat \phi (\omega) =
\int_{\mathbb{R}}\phi (x)e^{-2\pi i x \omega}d x$ is given by
\begin{equation}\label{prod2}
\widehat \phi (\omega) =  \prod_{k=1}^{\infty} p_k(2^{-k}\omega),
\quad \omega \in \RR\,.
\end{equation}
If $\phi \in C^\ell(\mathbb{R})$, then $\widehat \phi (\omega) =
o(\omega^{-\ell})$ as $\omega \to \infty$. Since the refinable
function of the limit mask~$\mathbf{a}$ is stable, it follows that
its symbol~$a_*(z)$ has no symmetric roots on the unit circle. The
claim follows by  Proposition~\ref{p30}. Indeed, by definition of
$p_k$, we get by \eqref{eq:aux}
$$
 \max\limits_{j=0, \ldots , \ell } 2^{-jk} |D^j a_*^{(k)}(-1)| = o(2^{-\ell k}) \quad \hbox{as}  \quad k \to \infty,
$$
which completes the proof.
\end{proof}



\begin{thebibliography}{}

\bibitem{BerWan} M.~A. Berger and Y.~Wang,
Bounded semigroups of matrices, Linear Alg. Appl., 166 (1992) 21-27.

\bibitem{BR} A. Ben-Artzi and A. Ron, Translates of exponential box splines and their related
spaces, Trans. Amer. Math. Soc., 309 (1988) 683-710.

\bibitem{Burkhart10} D. Burkhart, B.  Hamann and G. Umlauf, Iso-geometric finite element
analysis based on Catmull-Clark subdivision solids, Computer Graphics Forum,  29  (2010)
1575-1584.

\bibitem{Cabrelli} C.~A.~Cabrelli, C.~Heil and U.~M.~Molter, Self-similarity and
multiwavelets in higher dimensions, Memoirs Amer. Math. Soc., 170
(2004), No 807.

\bibitem{CaravettaDahmenMicchelli} A.~S.~Cavaretta, W.~Dahmen and C.~A.~Micchelli, Stationary
Subdivision, Mem. Amer. Math. Soc., 453 (1991) i-vi; 1-185.

\bibitem{Chaikin} G. M. Chaikin, An algorithm for high speed curve generation,
Comput. Gr. Image Process., 3 (1974) 346-349.

\bibitem{Charina} M. Charina, Vector multivariate subdivision schemes:
Comparison of spectral methods for their regularity analysis, App. Comp. Harm.
Anal., 32 (2012) 86-108.

\bibitem{CCGP_arxiv} M. Charina, C. Conti, N. Guglielmi and V. Yu. Protasov,
Regularity of non-stationary multivariate subdivision, arXiv:1406.7131

\bibitem{CharinaContiJetterZimm11} M. Charina, C. Conti, K. Jetter and
Georg Zimmermann, Scalar multivariate subdivision schemes and box
splines, Comput. Aided Geom. Design, 28 (2011) 285-306.

\bibitem{CharinaContiRomani13} M. Charina, C. Conti and L. Romani, Reproduction
of exponential polynomials by multivariate non-stationary
subdivision schemes with a general dilation matrix, Numer. Math.,
127 (2014) 223-254.

\bibitem{CharinaContiSauer2005} M. Charina, C. Conti, T. Sauer,
Regularity of multivariate vector subdivision schemes, Numerical
Algorithms, 39 (2005) 97-113.

\bibitem{ChenJiaRie} D.-R. Chen, R.-Q. Jia and S. D. Riemenschneider, Convergence of vector
subdivision schemes in Sobolev spaces, Appl. Comput. Harmon. Anal., 12 (2002) 128-149.


\bibitem{COS00} F. Cirak, M. Ortiz, P. Schr\"oder, Subdivision surfaces: A new paradigm for
thin-shell finite-element analysis, Int. J. Num. Meth. Eng., 47 (2000) 2039-2072.

\bibitem{CSAOS02} F. Cirak, M. J. Scott, E. K. Antonsson, M. Ortiz, P. Schr\"oder, Integrated modeling,
finite-element analysis, and engineering design for thin-shell structures
using subdivision, Comput. Aided Design, 34 (2002) 137-148.

\bibitem{CohenDyn} A. Cohen and N. Dyn, Nonstationary subdivision schemes and
multiresolution analysis, SIAM J. Math. Anal., 27 (1996) 1745-1769.

\bibitem{CollelaH} D.~Collela and C.~Heil, Characterization of scaling functions:
continuous solutions, SIAM~J. Matrix Anal. Appl., 15 (1994)
496-518.

\bibitem{Conti08} C. Conti, Stationary and nonstationary affine combination of
subdivision masks, Mathematics and Computers in Simulation, 81
(2010) 623-635.

\bibitem{ContiDynManniMazure13} C. Conti, N. Dyn, C. Manni and M.-L. Mazure,
Convergence of univariate non-stationary subdivision schemes via
asymptotic similarity, $arxiv.org/abs/1410.2729$.

\bibitem{ContiGRomani} C.~Conti, L. Gemignani and L.~Romani, Exponential splines and pseudo-splines:
generation versus reproduction of exponential polynomials, arXiv:1404.6624v2.

\bibitem{ContiRomani11} C.~Conti and L.~Romani, Algebraic conditions on non-stationary subdivision symbols for
exponential polynomial reproduction, J. Comp. Appl. Math., 236
(2011) 543-556.

\bibitem{ContiRomani13} C.~Conti and L.~Romani, Dual univariate m-ary subdivision
schemes of de Rham-type, J. Math. Anal. Appl., 407 (2013) 443-456.

\bibitem{CRY} C.~Conti, L.~Romani and J. Yoon, Sum rules versus approximate sum rules in subdivision,
$arxiv.org/abs/1411.2114$.

\bibitem{Conway} J. B. Conway, {\em Functions of One Complex Variable}, Springer 2001.


\bibitem{ISO} J. A. Cottrell, Th. J. R. Hughes and Y. Bazilevs,
{\it Isogeometric Analysis: Toward Integration of CAD and FEA}, John Wiley and Sons, 2009.

\bibitem{DM97} W. Dahmen and C. A. Micchelli, Biorthogonal wavelet expansions, Const.
Approx., 13 (1997) 293-328.

\bibitem{D} I.\,Daubechies, {\em Ten lectures on wavelets}, CBMS-NSR Series in Appl.
Math, 1992.

\bibitem{DL1992} I.~Daubechies and J.~C.~Lagarias, Sets of matrices all
infinite products of which converge, Linear Algebra Appl., 162
(1992) 227-263.


\bibitem{Unser1} R. Delgado-Gonzalo, P. Thevenaz, M. Unser, Exponential splines and
minimal-support bases for curve representation,  Comput. Aided
Geom. Design, 29 (2012) 109-128.

\bibitem{Unser2}  R. Delgado-Gonzalo, P. Thevenaz, C.S. Seelamantula, M. Unser,
Snakes with an ellipse-reproducing property, IEEE Transactions on
Image Processing, 21 (2012) 1258-1271.

\bibitem{Unser3} R. Delgado-Gonzalo,  M. Unser,  Spline-based
framework for interactive segmentation in biomedical imaging,
IRBM Ingenierie et Recherche Biomedicale / BioMedical Engineering
and Research, 34 (2013) 235-243.

\bibitem{Rham} G. de Rham, Sur une courbe plane, J. Math. Pures Appl., 35  (1956) 25-42.


\bibitem{DesDubuc} G. Deslauriers and S. Dubuc, Symmetric iterative interpolation
processes, Constr. Approx., 5 (1989) 49-68.

\bibitem{DKLR} N.\,Dyn, O.\,Kounchev, D.\,Levin, and H.\,Render, Regularity of
generalized Daubechies wavelets reproducing exponential
polynomials with real-valued parameters, Appl. Comput. Harmon.
Anal., 37 (2014) 288-306.

\bibitem{DL2002} N. Dyn and D. Levin, Subdivision schemes in geometric modelling,
Acta Numer., 11 (2002), 73-144.

\bibitem{DynLevin} N. Dyn and D. Levin, Analysis of asymptotic equivalent binary
subdivision schemes, J. Math. Anal. App., 193 (1995) 594-621.

\bibitem{Dyn2}  N. Dyn, D. Levin, A. Luzzatto, Exponentials reproducing subdivision scheme,
Found. Comput. Math., 3 (2003) 187-206.


\bibitem{DynLevinYoon} N. Dyn, D. Levin and J. Yoon,
Analysis of Univariate Nonstationary Subdivision Schemes
with Application to Gaussian-Based Interpolatory Schemes, SIAM J. Math. Anal., 39 (2007)
 470-488.


\bibitem{GoodmanLee} T. N. T. Goodman and S. L. Lee, Convergence of nonstationary
cascade algorithm, Num. Math., 84 (1999) 1-33.

\bibitem{GregoryDynLevin}
J. A. Gregory, N. Dyn and D. Levin, A butterfly subdivision scheme
for surface interpolation with a tension control, ACM Transactions
on Graphics, 2 (1990) 160-169.

\bibitem{Manni1} N. Guglielmi, C. Manni, D. Vitale,
Convergence analysis of $C^2$ Hermite interpolatory subdivision
schemes by explicit joint spectral radius formulas, Linear Algebra
Appl., 434 (2011) 884-902.

\bibitem{GuglielmiZennaro} N. Guglielmi and M. Zennaro,
On the asymptotic properties of a family of matrices, Linear Algebra Appl., 322 (2001)
169-192.

\bibitem{GugPro} N. Guglielmi and V. Yu. Protasov,
Exact computation of joint spectral characteristics of linear operators,
Found. Comput. Math., 13 (2013) 37-97.


\bibitem{JetterPlonka} K. Jetter and  G. Plonka, A survey on $L_2$-Approximation order from
shift-invariant spaces,
Multivariate Approximation and Applications (N. Dyn, D. Leviatan,
D. Levin, A. Pinkus, eds.), Cambridge University Press, 2001,
73-111.

\bibitem{J95} R.-Q. Jia, Subdivision schemes in $L_p$ spaces, Adv.
Comput. Math., 3 (1995) 309-341.

\bibitem{J98} R.-Q. Jia, Approximation properties of multivariate wavelets,
Math. Comput., 67 (1998) 647-665.

\bibitem{JiaJiang} R.-Q. Jia and Q.-T. Jiang,  Approximation power of refinable vectors of functions,
Stud. Adv. Math., 25, Amer. Math. Soc., Providence, RI, 2002,
155-178.

\bibitem{JiaLei} R. Q. Jia and J. J. Lei, Approximation by piecewise exponentials,
SIAM J. Math. Anal., 22 (1991), 1776-1789.

\bibitem{H03} B. Han, Vector cascade algorithms and refinable function vectors in
Sobolev spaces, J. Approx. Theory, 124 (2003) 44-88.

\bibitem{Han} B.~Han, Nonhomogeneous wavelet systems in high dimensions, Appl.
Comput. Harmon. Anal., 32 (2012) 169-196.

\bibitem{Han2} B. Han, Z. Shen, Compactly supported symmetric $C^\infty$ wavelets with
spectral approximation order, SIAM J. Math. Anal., 40 (2008)
905-938.

\bibitem{Unser4} I. Khalidov, M. Unser, J. P. Ward, Operator-like wavelet bases of $L_2(\RR^d)$,
J. Fourier Anal. Appl., 19 (2013) 1294-1322.

\bibitem{Yoon1} Y.-J. Lee, J. Yoon, Non-stationary subdivision schemes for surface
interpolation based on exponential polynomials, Appl. Numer.
Math., 60 (2010), 130-141.

\bibitem{Loop} C. Loop, Smooth subdivision surfaces based on triangles, M.S.
Mathematics thesis, University of Utah, 1987.

\bibitem{LycheMaz} T. Lyche and M-L. Mazure, On the existence of
piecewise exponential B-splines, Advances Comput. Math., 25 (2006)
105-133.

\bibitem{LycheMer1} T. Lyche, J-L. Merrien, Hermite
subdivision with shape constraints on a rectangular mesh, BIT, 46
(2006) 831-859.

\bibitem{LycheMer2} T. Lyche and J-L. Merien, $C^1$ Interpolatory
subdivision with shape constraints for curves, Siam J. Numer.
Anal.,  44 (2006) 1095-1121.


\bibitem{Manni3} C. Manni, M.-L. Mazure, Shape Constraints and optimal bases for
$C^1$ Hermite interpolatory subdivision schemes, SIAM J. Num.
Anal., 48 (2010) 1254-1280.

\bibitem{MicSauer97} C.A. Micchelli and T. Sauer, Regularity of multiwavelets,
Adv. Comput. Math., 7 (1997) 455–545.


\bibitem{P3} I. Ya. Novikov, V. Yu. Protasov and M. A. Skopina, {\em Wavelet
theory}. Translated from the Russian original by Evgenia Sorokina.
Translations of Mathematical Monographs, 239, American
Mathematical Society, Providence, RI, 2011.

\bibitem{P} V.Yu.\,Protasov,  Extremal $L_p$-norms of linear operators and self-similar
functions, Linear Alg. Appl., 428 (2008) 2339-2356.

\bibitem{P1} V. Yu. Protasov, On the growth of products of trigonometric
polynomials, Math. Notes, 72 (2002) 819-832.

\bibitem{P2} V. Yu. Protasov, Spectral factorization of 2-block Toeplitz matrices
and refinement equations, St. Petersburg Math. Journal, 18 (2004)
607-646.

\bibitem{Romani} L. Romani, From approximating subdivision schemes for exponential
splines to high-performance interpolating algorithms, J. Comput.
Appl. Math., 224 (2009) 383-396.

\bibitem{Ron_exp_box_splines} A. Ron, Exponential box splines, Constr. Approx., 4
(1988) 357-378.

\bibitem{RotaStrang} G.-C.~Rota and G.~Strang, A note on the joint spectral radius,
Indag. Math., 22 (1960) 379-381.

\bibitem{Unser6} V. Uhlmann, R. Delgado-Gonzalo, C. Conti, L. Romani, M. Unser,
Exponential Hermite splines for the analysis of biomedical images,
{\sl Proceedings of IEEE International Conference on Acoustic,
Speech and Signal Processing (ICASSP)}, 2014 1650-1653.

\bibitem{Unser5} C. Vonesch, T. Blu, M. Unser, Generalized Daubechies wavelet families,
IEEE Trans. Signal Process., 55 (2007) 4415-4429.

\bibitem{W} J.\,Warren, H. Weimer, {\em Subdivision methods for geometric design},
Morgan-Kaufmann, 2002.

\bibitem{Z} A.\,Zygmund, {\em Trigonometric series}, Third edition, Cambridge
mathematical library, Cambridge University Press, 2002.

\end{thebibliography}
\end{document}